\documentclass[reqno]{amsart}
\usepackage{hyperref}

\AtBeginDocument{{\noindent\small
\emph{Electronic Journal of Differential Equations},
Vol. 2022(2022), No. $63$ , pp. 1--25.\newline
ISSN: 1072-6691. URL: http://ejde.math.txstate.edu or http://ejde.math.unt.edu}
\thanks{\copyright 2020. This work is licensed under a CC BY 4.0 license}
\vspace{8mm}}

\begin{document}
\title[\hfilneg EJDE-2022/**\hfil Positively Homogeneous Maximal   Monotone Operators]
{Solvability of  Inclusions Involving Perturbations of Positively Homogeneous Maximal\newline  Monotone Operators}

\author[Adhikari, Aryal, Bhatt, Kunwar,  Puri,  Ranabhat \hfil EJDE-2022/**\hfilneg]
{Dhruba R. Adhikari, Ashok Aryal, Ghanshyam Bhatt, Ishwari J. Kunwar,  Rajan Puri, Min Ranabhat}

\address{Dhruba R. Adhikari \newline
Department of Mathematics,
Kennesaw State University,
Marietta, GA 30060, USA}
\email{dadhikar@kennesaw.edu}

\address{Ashok Aryal \newline
Mathematics Department, Minnesota State University Moorhead, Moorhead, MN 56563}
\email{ashok.aryal@mnstate.edu}

\address{Ghanshyam Bhatt\newline
	Department of Mathematical Sciences, Tennessee State University,  Nashville, TN 37209}
\email{gbhatt@tnstate.edu}

\address{Ishwari J. Kunwar \newline
Department of Mathematics and Computer Science, Fort Valley State University, Fort Valley, GA 31030}
\email{kunwari@fvsu.edu}

\address{Rajan Puri \newline
	Department of Mathematics, Wake Forest University,	 Winston-Salem, NC 27109}
\email{purir@wfu.edu}

\address{Min Ranabhat\newline
Department of Mathematical Sciences, University of Delaware, EWG 315, Newark, DE 19716}
\email{ranabhat@udel.edu}

\thanks{Submitted  April 1, 2022. Published Aug 30, 2022.}
\keywords{Topological degree theories;
	operators of type $(S_+)$;    maximal \hfill\break\indent  monotone  operators;  duality mappings; Yosida approximants}

\begin{abstract}
 	Let $X$ be a real reflexive Banach space and  $X^*$  be its   dual space. Let $G_1$ and $G_2$ be open subsets of $X$  such that 
 $\overline G_2\subset G_1$,   $0\in G_2$,
 and $G_1$ is bounded.
 Let $L: X\supset D(L)\to X^*$ be a densely defined linear maximal
 monotone operator,   $A:X\supset D(A)\to 2^{X^*}$  be a  maximal monotone  and positively
 homogeneous operator of degree $\gamma>0$,   $C:X\supset D(C)\to X^*$ be a bounded
 demicontinuous operator of type $(S_+)$ w.r.t.  $D(L)$, and $T:\overline G_1\to 2^{X^*}$ be a compact and upper-semicontinuous operator whose values are  closed and convex  sets in $X^*$. We first take $L=0$ and establish the existence of nonzero solutions of 
 $Ax+ Cx+ Tx\ni  0$ 	in the set $G_1\setminus G_2.$   Secondly, we assume that $A$ is bounded and establish the existence of nonzero solutions of $Lx+Ax+Cx\ni 0$ in   $G_1\setminus G_2.$
  	We  remove the restrictions  $\gamma\in (0, 1]$ for $Ax+ Cx+ Tx\ni 0$ and $\gamma= 1$ for $Lx+Ax+Cx\ni 0$  from  such existing results in the literature.  We also present applications to
 elliptic and parabolic partial differential equations in general
 divergence form  satisfying Dirichlet boundary conditions. 
\end{abstract}

\maketitle
\numberwithin{equation}{section}
\newtheorem{theorem}{Theorem}[section]
\newtheorem{lemma}[theorem]{Lemma}
\newtheorem{corollary}[theorem]{Corollary}
\newtheorem{definition}[theorem]{Definition}
\newtheorem{proposition}[theorem]{Proposition}
\newtheorem{remark}[theorem]{Remark}
\newtheorem{application}[theorem]{Application}
\allowdisplaybreaks

\newcommand{\la}{\lambda }
\newcommand{\ol}{\overline}
\newcommand{\lan}{\langle }
\newcommand{\ran}{\rangle }
\newcommand{\var}{\varepsilon}
\newcommand{\wi}{\widetilde}
\newcommand{\rh}{\rightharpoonup}
\newcommand{\nifty}{n\to\infty}
\newcommand{\ls}{\limsup_{n\to\infty}}
\newcommand{\li}{\liminf_{n\to\infty}}
\newcommand{\vp}{\varphi}
\renewcommand{\theenumi}{\roman{enumi}}%

\section{Introduction and Preliminaries}\label{S1}

	Let $X$ be a real reflexive Banach space and  $X^*$  be its  topological dual space. The symbol $2^{X^*}$ denotes the collection of all subsets of $X^*$.  
The norm on $X$ is  denoted by $\|\cdot\|_X$.  When there is no  risk of misunderstanding, the norms  on $X$ and $X^*$ are both denoted  by $\|\cdot\|$.  The pairing $\langle x^*,x\rangle$ denotes the value of the
functional $x^*\in X^*$ at $x\in X$.
The symbols $\partial Z, \textrm{Int} Z, \overline Z$ and $ \operatorname{co}Z$ denote
the  boundary, interior,  closure, and convex hull of the set $Z \subset X,$  respectively.
The symbol $B_X(0,R)$ denotes the open ball of radius $R>0$ with center at
$0$ in  $X$. The symbols $\mathbf{R}$ and $\mathbf{R}_+$ denote $ (-\infty, \infty)$ and
$[0,\infty)$, respectively. For a sequence $\{x_n\}$  in $X$ and $x_0\in X$,  we denote  by $x_n\to x_0$ and $x_n\rightharpoonup x_0$  the  strong convergence and weak convergence,  respectively.  Given another real Banach $Y,$ an operator $T : X\supset D(T)\to Y$ is said
to be \textit{bounded} if it maps bounded subsets of the domain $D(T)$ onto
bounded subsets of  $Y.$
The operator $T$ is said to be \textit{compact} if it maps bounded subsets of
$D(T)$ onto relatively compact subsets in $Y.$
The operator $T$  is said to be \textit{demicontinuous}
if it is strong-to-weak continuous on $D(T)$.

A multivalued operator $A$ from $X$ to $X^*$ is written as
$A:X\supset D(A) \to 2^{X^*},$ where $D(A)=\{x\in X: Ax\neq\emptyset\}$
is the effective domain of $A$. Here,  $Ax$ means $A(x)$, and these notations are used interchangeably in the sequel. 
We denote  the graph of $A$ by ${\rm Gr}(A)$, i.e., ${\rm Gr}(A)=\{(x,y): x\in D(A), y\in Ax\}$.

\begin{definition}\label{Hom} \rm
	An operator $A: X\supset D(A)\to 2^{X^*}$ is said to be
	\textit{positively homogeneous of degree $\gamma >0$} if $(x,y)\in {\rm Gr}(A)$ implies $sx\in D(A)$ for all
	$s\ge 0$ and $(sx, s^\gamma y)\in {\rm Gr}(A).$ 
\end{definition}

\begin{remark}\rm 
	An equivalent condition for an operator $A: X\supset D(A)\to 2^{X^*}$to be
	positively homogeneous of degree $\gamma >0$ is that $x\in D(A)$ implies $sx\in D(A)$ for all $s\ge 0$  and $s^\gamma Ax \subset A(sx).$  It follows that  a positively homogeneous operator $A$ of degree $\gamma >0$ satisfies $0\in A(0)$.  When $A$ is positively homogeneous of degree $\gamma >0,$ it can be verified that $x\in D(A)$ implies $sx\in D(A)$  for all $s>0$ and $s^\gamma Ax =  A(sx).$ However, in general, the property $s^\gamma Ax =  A(sx)$ may not be true for $s=0.$ For example, let $A: \mathbf R\supset {[0, \infty)} \to 2^\mathbf R$ be given by 	$$A x=\begin{cases} 
		(-\infty, 0] & \mbox{for  }
		x=0\\
		x^\gamma & \mbox{for  }
		x>0.
	\end{cases}
	$$  Clearly, $A(0) = (-\infty,  0]\ne \{0\}.$  
\end{remark}

A \textit{gauge} function is a strictly  increasing continuous function $\varphi: \mathbf{R}_+ \to \mathbf{R}_+$ with $\varphi(0) = 0$ and $\varphi(r) \to \infty$ as $r\to\infty$.  The  \textit{duality 	mapping} of $X$ corresponding to a gauge function $\varphi $ is the mapping  $J_\varphi:X \supset D(J_\varphi )\to 2^{X^*}$  defined by
$$
J_\vp x = \{x^*\in X^* : \langle x^*, x\rangle
= \varphi(\|x\|)\|x\|, \; \|x^*\| = \varphi(\|x\|)\}, \quad x\in X.
$$
The Hahn-Banach theorem ensures that $D(J_\varphi) = X,$ and therefore  $J_\varphi :X\to 2^{X^*}$
is, in general, a multivalued mapping. The  duality mapping corresponding to the gauge function $\varphi (r) = r$ is called the \textit{normalized duality} mapping and denoted  by $J$. 
It is well-known that the duality mapping $J_\vp$ 	
satisfies
$$J_\varphi x  = \frac{\vp(\|x\|)}{\|x\|} Jx, \quad  x\in X\setminus \{0\}.
$$
Since $J$ is homogeneous of degree $1$,  we have
\begin{equation*}
	J_\varphi(s x) = \frac{\vp(s\|x\|)}{\|x\|} Jx, \quad (s, x)\in \mathbf{R}_+\times (X\setminus \{0\}).
\end{equation*}
In particular, when $\vp(r) = r^{p-1}, 1< p< \infty$, we get  	$J_\varphi x = \|x\|^{p-2} Jx,$  $x\in X\setminus \{0\},$ which implies
\begin{equation*}
	J_\varphi(s x) = s^{p-1} J_\vp  x, \quad (s, x)\in \mathbf{R}_+\times X,
\end{equation*}
i.e., $J_\varphi$ is positively homogeneous of degree $p-1$. 

When $X$ is reflexive and both  $X$ and $X^*$ are strictly convex,   the inverse $J^{-1}_\vp$ of $J_\vp$ is the duality mapping  of $X^*$ with the gauge function $\vp^{-1}(r) = r^{q-1},$ where $q$ is given by $1/p + 1/q = 1.$  It is easy to verify that
\begin{equation}\label{11}
	J_\varphi^{-1}(s x^*) = s^{q-1} J_\varphi^{-1} x^*, \quad (s, x^*)\in \mathbf{R}_+\times X^*.\end{equation}
It is clear that $J_\vp$ is positively homogeneous of degree $\gamma>0$ if and only if $\vp$ is positively homogeneous of degree $\gamma>0.$   Additional  properties of duality mappings in connection with Banach space geometry  can be found in Alber and Ryazantseva~\cite{Alber} and Cioranescu~\cite{Cioranescu}. 
\begin{definition}{\rm 
	An operator $A:X\supset D(A)\to 2^{X^*}$
	is said to be \textit{monotone} if for all
	$(x, u), (y, v)\in  {\rm Gr}(A)$ we have
	$
	\langle u - v,x-y\rangle \ge 0.
	$
	A monotone operator $A:X\supset D(A)\to 2^{X^*}$ is said to be \textit{maximal monotone} if ${\rm Gr}(A)$ is
	maximal in $X\times X^*$, when $X\times X^*$ is partially ordered by 
	set inclusion.}
\end{definition}
In what follows, we assume that $X$ is reflexive and both $X$ and $X^*$ are strictly convex.  
It is well-known that the duality mapping $J_\vp$ is maximal monotone. A monotone operator $A$ is maximal if and only if $R(A+\lambda J_\vp) = X^*$ for all $\lambda\in (0,\infty)$ and all  gauge functions $\vp.$ For a proof of this result for  $\vp(r) = r^{p-1}, 1<p<\infty,$ the reader is referred, for example,  to Barbu~\cite[Theorem~2.3]{BA}.  

\begin{definition}\label{def1} \rm
	Let $L:X\supset D(L)\to X^*$ be a densely defined
	linear maximal monotone operator. An operator   $C: X\supset D(C)\to X^*$ is said to be of \textit{type $(S_+)$} \textit{w.r.t.} $D(L)$ if for
	every sequence $\{x_n\}\subset D(L)\cap D(C)$ with $x_n\rightharpoonup x_0$
	in $X$, $Lx_n\rightharpoonup Lx_0$ in $X^*$ and
	$$
	\limsup_{n\to\infty} \langle Cx_n, x_n- x_0\rangle \le 0,
	$$
	we have $x_n\to x_0$ in $X.$ In this case, if $L=0$, then $C$ is said to be of \textit{type $(S_+)$}.
\end{definition}

\begin{definition}\label{def2} \rm
	A family of operators  $C(s):X\supset G\to X^*, s\in[0,1]$, is said to be
	a  \textit{homotopy of type $(S_+)$ w.r.t. $D(L)$} if for every sequence
	$\{x_n\}\subset D(L)\cap G$ with $x_n\rightharpoonup x_0$ in $X$
	and $Lx_n\rightharpoonup Lx_0$ in $X^*$, $\{s_n\}\subset[0,1]$ with $s_n\to s_0$ and
	$$
	\limsup_{n\to\infty} \langle C(s_n)x_n, x_n - x_0\rangle \le 0,
	$$
	we have $x_n\to x_0$ in $X,$ $x_0\in G$ and $C(s_n)x_n \rightharpoonup C(s_0)x_0$ in
	$X^*.$ In this case, if $L=0$, then $C(s)$ is said to be a \textit{homotopy of type $(S_+)$.}
	A homotopy $C(s)$ of type $(S_+)$ w.r.t. $D(L)$ is \textit{bounded} if the set
$\{C(s)x: s\in[0,1], x\in G\}$
	is bounded.
\end{definition}

\begin{definition}\label{D0} \rm
	An operator $T: X\supset D(T)\to 2^{X^*}$ is
	said to be of \textit{class $(P)$} if  
	
	\begin{enumerate}
		\item it maps bounded sets to
		relatively compact sets; 
		\item for every $x\in D(T)$, $Tx$ is a closed and
		convex subset of $X^*;$ and
		\item  $T(\cdot)$ is
		upper-semicontinuous, i.e., for every closed set
		$F\subset X^*$, the set $T^-(F) = \{x\in D(T): Tx \cap F \ne
		\emptyset\}$ is closed in $X$.
	\end{enumerate} 
\end{definition}
Hu and Papageorgiou   introduced the operators of class $(P)$ in \cite{HP}. We recall a compact-set valued upper-semicontinuous operator $T$ is
closed. Furthermore,  given an operator $T$ of class $P$ and a sequence
$\{(x_n, y_n)\}\subset {\rm Gr}(T)$ such that $x_n\to x\in D(T)$, the sequence
$\{y_n\}$ has a cluster point in $Tx$.

This paper is organized as follows. In Section~\ref{S2}, we study  variants of  the standard Yosida approximants  introduced  in Br\'ezis, Crandall, and Pazy~\cite{BCP}  and their fundamental properties.  
Since the topological degree theory for $(S +)$-operators is employed to establish the main existence results in  the later sections, we provide  several results involving  variants of Yosida approximants related to the Browder  degree theory \cite{Browder1983}.  

In Section~\ref{S3},  we first prove the existence of nonzero solutions of $Ax+Cx+ Tx\ni 0$  by utilizing the topological degree theories developed by
Browder \cite{BrowderHess1972} and Skrypnik \cite{Skrypnik1994}.
In this case, $A$ is  maximal monotone with $A(0) =\{0\}$  and positively
homogeneous of degree $\gamma>0$,  $C$ is bounded demicontinuous
of type $(S_+)$, and $T$ is  of class $(P)$. 
This result extends an analogous  result  for  $\gamma\in (0, 1]$ established  in \cite{Adhikari2017} to an arbitrary degree of homogeneity $\gamma>0.$
Another main result established in this section is the existence of nonzero solutions of
$Lx+Ax+Cx\ni 0$, where $L$, $C$ are as above, and $A$ is a  bounded
maximal monotone  and positively homogeneous of degree $\gamma>0$.
This result  extends  an analogous result for $\gamma =1$   established  in \cite{Adhikari2017} to an arbitrary degree of homogeneity $\gamma>0.$

In Section~\ref{S4}, we present some applications of the theories developed in Section~\ref{S3}  to elliptic and parabolic partial differential equations, in general, divergence form that  include $p$-Laplacian with $1<p<\infty$ and satisfy Dirichlet boundary conditions.

For additional facts and various topological degree theories related
to the subject of this paper, the reader is referred to Adhikari and Kartsatos \cite{AK1, AK2008}, Kartsatos and Lin \cite{KartsatosLin2003}, and
Kartsatos and Skrypnik \cite{KartsatosSkrypnik,KartsatosSkrypnik2005b}.
For further information on functional analytic tools used herein,
the reader is referred to Barbu \cite{BA}, Browder \cite{Browder1976},
Pascali and Sburlan \cite{PS}, Simons \cite{Simons},
Skrypnik \cite{Skrypnik1986,Skrypnik1994}, and Zeidler \cite{Zeidler1}.

\section{Variants of Yosida Approximants  and Related Properties}\label{S2}
Let $X$ be a  strictly convex and reflexive Banach space with strictly
convex  $X^*.$ By using the duality mapping $J_\vp$ corresponding to an arbitrary gauge function $\vp,$ we  study variants of  the  Yosida approximants in Br\'ezis et al. in \cite{BCP} and resolvents  of  a maximal monotone operator $A:X\supset D(A)\to 2^{X^*}.$ 
For each $\lambda>0$ and each $x\in X$,  the inclusion 
\begin{equation} \label{Yosida-inclusion}
	0\in J_\vp(x_\lambda- x) + \lambda Ax_\lambda 
\end{equation} has a unique solution $x_\lambda\in D(A)$ (see Proposition~\ref{prop0} (i)).  
We define $J_\lambda^\vp : X\to D(A)\subset  X$  and $A_\lambda^\vp : X\to X^*$ by 
\begin{equation}\label{Yosida}
	J_\lambda^\vp x := x_\lambda \quad \mbox{ and }\quad A_\lambda^\vp x: = \dfrac{1}{\lambda} J_\vp (x- J_\lambda^\vp x), \quad x\in X.
\end{equation} The operators $A_\lambda^\vp$ and $J_\lambda^\vp$ are  variants of the  standard Yosida approximant $A_\lambda$  and resolvent $J_\lambda$  of $A.$  
For each  $x\in X$,  we have $$A_\lambda^\vp x\in A (J_\lambda^\vp x) \quad \mbox{and} \quad 	x = J_\lambda^\vp x + J^{-1}_\vp ( \la A_\lambda^\vp x).$$  When $\vp(r) = r^{p-1},$ a splitting of $x$  in terms of $A_\lambda^\vp$ and $J_\lambda^\vp$ is 
\begin{equation}\label{splitting}
	x = J_\lambda^\vp x + \lambda^{q-1} J^{-1}_\vp ( A_\lambda^\vp x),
\end{equation}  and therefore
\begin{equation}\label{closed-form}
	A_\lambda^\vp x= \left(A^{-1}+ \lambda^{q-1} J_\vp^{-1}\right)^{-1} x, \quad x\in X.
\end{equation}
It is easy to verify  that $A= A_\lambda^\vp$ if and only if $A=0$. In fact, if $A= 0$, then $J_\lambda^\vp  = I$, the identity operator on $X$. Moreover, if $0\in D(A)$ and $0\in A(0)$, then  $A_\lambda^\vp 0 =0$.	 

The choice of an appropriate gauge function is essential for the main existence results in this paper. The following proposition summarizes some important properties of $A_\lambda^\vp$ and $J_\lambda^\vp$ along the lines of analogous  properties of  $A_\lambda$ and $J_\lambda$.  A complete proof is provided here for the reader's convenience.

\begin{proposition}\label{prop0} Let $X$ be a   strictly convex and reflexive Banach space with strictly
	convex dual $X^*$ and
	$A:X\supset D(A)\to 2^{X^*}$ be  a maximal monotone operator.  Then the following statements hold.
	
	\begin{enumerate}
		\item  The operator  $A_\lambda^\vp$ is single-valued, monotone,  bounded on bounded subsets of $X,$ and demicontinuous  from $X$ to $X^*$.
		\item   For every $x\in D(A)$ and $\lambda >0,$ we have $$\|A_\lambda^\vp x\|\le |Ax|:= \inf\{\|x^*\|: x^*\in Ax \}.$$
		\item The operator $J_\lambda^\vp $ is bounded on bounded subsets of $X,$  demicontinuous  from $X$ to $ D(A)$, and 
		\begin{equation*}\label{resolvent}
			\lim_{\lambda\to 0} J_\lambda^\vp x = x  \quad  \mbox{for all } x\in\overline{\mbox{co}D(A)}.
		\end{equation*}
		\item If $\lambda_n\to 0$, $x_n\rh x$ in $X$, $A_{\lambda_n}^\vp x_n \rh y$ and 
		$$\limsup_{n, m\to\infty} \lan A_{\lambda_n}^\vp x_n- A_{\lambda_m}^\vp x_m , x_n - x_m\ran \le 0,  $$
		then $(x , y) \in {\rm Gr} (A)$ and 
		$$\lim_{n, m\to\infty} \lan A_{\lambda_n}^\vp x_n- A_{\lambda_m}^\vp x_m , x_n - x_m\ran =0.  $$
		\item For every  sequence $\{\lambda_n\}$ with  $\lambda_n\to 0$, $A_{\lambda_n}^\vp x\rh A^{\{0\}}x$ for all $x\in D(A)$.  In addition, if $X^*$ is uniformly convex,  then $A_{\lambda_n}^\vp x \to A^{\{0\}}x$  for all $x\in D(A).$
		\item If $\lambda_n\to 0$ and  $x\not\in\overline{D(A)},$ then
		$$\lim\limits_{n\to\infty} \|A_{\lambda_n}^\vp x\| = \infty.$$
	\end{enumerate}
\end{proposition}

\begin{proof}
	(i)   We first show that $J_\lambda^\vp$ is single-valued. Given $x\in X$ and $\lambda>0,$ let $x_\lambda$ and $\tilde x_\lambda$  be solutions of  (\ref{Yosida-inclusion}). Take $y \in Ax_\lambda$ and  $\tilde y \in A\tilde x_\lambda$ such that 
	$$ J_\vp(x_\lambda- x) + \lambda y=0\quad \mbox{ and } \quad J_\vp(\tilde x_\lambda- x) + \lambda \tilde y=0.$$
	This along with the monotonicity of $A$ and $J_\vp$ implies 
	\begin{equation}\label{21}
		\lan  J_\vp(x_\lambda- x) - J_\vp(\tilde x_\lambda- x), (x_\lambda- x)-(\tilde x_\lambda- x)\ran = 0.
	\end{equation}
	Since $X$ is strictly convex,  it follows that $J_\vp$  is strictly monotone, i.e, for  $u_1, u_2\in X,$ we have
	$$\lan J_\vp u_1 - J_\vp u_2, u_1- u_2\ran > 0 \mbox{ if and only if } u_1 \ne u_2.$$  It  follows from (\ref{21}) that  
	$x_\lambda =\tilde x_\lambda.$ Thus, $J_\la^\vp$ is single-valued, and therefore  $A_\la^\vp$ is also single-valued.  It is easy to verify the monotonicity of  $ A_\la^\vp.$ 
	
	To show $A_\la^\vp$ is bounded, let $B\subset X$ be bounded. For each $x\in B$, let $x_\la = J_\la^\vp x.$ Let $(u, v) \in {\rm Gr} (A).$ Using (\ref{Yosida-inclusion}), it follows that
	$$\lan J_\vp(x_\la-x)+\la y_\la, x_\la-u\ran = 0,$$
	where $y_\la\in Ax_\la.$ This implies 
	$$ \lan J_\vp(x_\la-x), x_\la-u\ran = - \la \lan y_\la, x_\la - u\ran \le \la \lan v, u-x_\la \ran. $$ The last inequality follows from the monotonicity of $A$. It then  follows that 
	\begin{eqnarray}\label{22}
		\begin{aligned}
			\lan J_\vp(x_\la-x), x_\la-x\ran&= \lan J_\vp(x_\la-x), x_\la-u\ran+ \lan J_\vp(x_\la-x), u-x\ran\\
			&\le\la \lan v, u-x_\la \ran+\lan J_\vp(x_\la-x), u-x\ran\\
			&=\la \lan v, u-x \ran+\la \lan v, x-x_\la \ran+\lan J_\vp(x_\la-x), u-x\ran.
		\end{aligned}
	\end{eqnarray}
	This implies
	\begin{equation}\label{23}
		\vp(\|x_\la-x\|)\|x_\la-x\| \le \la\|v\| \left(\|u-x\|+\|x_\la-x\|\right) +\vp(\|x_\la-x\|)\|u-x\|.
	\end{equation}
	If $\{x_\la : x\in B\}$ is unbounded,  the inequality (\ref{23}) yields a contradiction.  Thus, $J_\la^\vp$ is bounded on $B.$  Since $J_\vp$ is bounded on $B$, it follows from (\ref{Yosida}) that $A_\la^\vp$ is also bounded on $B$. 
	
	Let $\{x_n\}\subset X$ be such that $x_n\to x_0\in X$ as $n\to \infty.$  Denote $u_n  = J_\la^\vp x_n$ and $v_n = A_\la^\vp x_n,$ so that 
	\begin{equation} 
		\label{Yosida2}J_\vp (u_n - x_n) +\la v_n = 0.
	\end{equation}  
	Since  $J_\la^\vp$  and $A_\la^\vp$  are bounded on bounded sets, both $\{u_n\}$ and $\{v_n\}$ are bounded.  Since  $J_\vp$ and $A$ are monotone, it follows from   
	\begin{eqnarray*}
		&\lan J_\vp(u_n-x_n)- J_\vp(u_m-x_m), (u_n-x_n) - (u_m-x_m)\ran\\
		&= -\la  \lan v_n-v_m,(u_n-x_n) - (u_m-x_m)\ran
	\end{eqnarray*}
	that
	$$\lim_{n, m\to\infty} \lan v_n - v_m, u_n-u_m\ran = 0$$ and
	$$\lim_{n, m\to\infty} \lan J_\vp(u_n-x_n)- J_\vp(u_m-x_m), (u_n-x_n) - (u_m-x_m)\ran = 0.$$
	Passing to subsequences, we may assume that $u_n\rh u_0$ in $X$,  $v_n\rh v_0$ in $X^*,$ and $J_\vp(u_n-x_n) \rh w_0$ in $X^*$ for some $u_0\in X$ and some $v_0, w_0\in X^*.$ 	By   \cite[Lemma~2.3]{BA}, it follows that $(u_0, v_0)\in {\rm Gr}(A)$ and $(u_0-x_0, w_0)\in {\rm Gr}(J_\vp).$ Using all these in (\ref{Yosida2}), we get 
	$J_\vp(u_0-u_0) + \la v_0 = 0,$ which implies $u_0 = J_\la^\vp x_0$ and $v_0 = A_\la^\vp x_0,$ i.e., $ J_\la^\vp x_n \rh  J_\la^\vp x_0$ and $A_\la^\vp x_n \rh  A_\la^\vp x_0$ as $n\to\infty$. This proves the demicontinuity of $J_\la$  and $A_\la$. \\
	\vspace{-2ex}
	
	(ii)  Let $x\in D(A)$ and $\lambda>0$. Let $y\in Ax$ and $x_\la = J_\la^\vp x.$  Then
	\begin{eqnarray*}
		0&\le& \lan y - A_\la x, x- x_\la\ran\\
		&= & \lan y, x- x_\la\ran -\frac{1}{\la}\vp(\|x- x_\la\|) \|x- x_\la\|\\
		&\le& \|y\| \|x- x_\la\| - \frac{1}{\la}\vp(\|x- x_\la\|) \|x- x_\la\|,
	\end{eqnarray*}
	which implies 
	$\vp(\|x- x_\la\|) \le \la \|y\|,$  and therefore $$\|A_\la^\vp x\|= \dfrac{1}{\la}\|J_\vp(x- x_\la)\|\le \|y\|.$$ Consequently,
	we have
	$\|A_\la^\vp x\| \le |Ax|:= \inf\{\|y\|: y\in Ax\}.$\\
	\vspace{-2ex}
	
	(iii) The boundedness of $J_\la^\vp$ on bounded subsets of $X$  and its demicontinuity are already proved in  (i).  Let  $ x\in\overline{\textrm{co}D(A)}$ and $(u, v)\in {\rm Gr}(A).$  Following the arguments that lead to  (\ref{23}), we find that 
	$\{x_\la-x: \la>0\}$ is bounded, and  therefore $\{J_\vp(x_\la-x): \la>0\}$ is bounded.  Let $\{\la_n\}\subset (0, \infty)$ be such that $\la_n \to 0.$  Let $y\in X^*$ be such that $J_\vp({x_{\la_n}}-x)\rh y$ in $X^*$. 
	Then (\ref{22}) yields
	$$ \limsup_{n\to\infty} 	\vp(\|x_{\la_n}-x\|)\|x_{\la_n}-x\|  \le \lan y, u-x\ran.$$
	It is clear that this argument applies to all $u\in \overline{\textrm{co}D(A)}$. Taking $u= x$, we obtain
	$$ \lim_{n\to\infty} 	\vp(\|x_{\la_n}-x\|)\|x_{\la_n}-x\|  =0.$$ By the homeomorphic property of the gauge function $\vp$, it follows that we must have
	$x_{\la_n}\to x$ as $n\to\infty.$ This completes the proof of (iii). \\
	\vspace{-2ex}
	
	(iv)  Let $u_n = J_{\lambda_n}^\vp x_n$ for all $n$. 
	Since $\{A_{\lambda_n}^\vp x_n\}$ is bounded, it follows that 
	$$\vp(\|x_n - u_n\|) =\vp( \|x_n -J_{\lambda_n}^\vp x_n\|)= \|J_\vp(x_n -J_{\lambda_n}^\vp x_n) \|= \la_n \|A_{\lambda_n}^\vp x_n\|\to 0$$
	as $n\to\infty$. This implies  $ \|x_n - u_n\|\to 0$ as $n\to\infty.$ 
	Since
	\begin{eqnarray*}
		\lan A_{\lambda_n}^\vp x_n- A_{\lambda_m}^\vp x_m, x_n - x_m \ran &=&\lan A_{\lambda_n}^\vp x_n- A_{\lambda_m}^\vp x_m, u_n - u_m \ran \\
		&&+\lan A_{\lambda_n}^\vp x_n- A_{\lambda_m}^\vp x_m, (x_n -u_n) - (x_m-u_m)\ran  
	\end{eqnarray*}
	and $A$ is monotone, it follows as in Br\'ezis et al.~\cite{BCP} that 
	$$\lim_{n, m\to\infty}\lan A_{\lambda_n}^\vp x_n- A_{\lambda_m}^\vp x_m, x_n - x_m \ran=0 \mbox{ and } \lim_{n, m\to\infty}\lan A_{\lambda_n}^\vp x_n- A_{\lambda_m}^\vp x_m, u_n - u_m\ran =0.$$ The conclusion of (iv) now follows from \cite[Lemma~2.3]{BA}. \\
	\vspace{-2ex}
	
	(v) Let $x\in D(A)$.  Since $X^*$   is reflexive and strictly convex and $Ax$ is a closed and convex subset of $X^*,$ it follows that there exists a unique element of  $Ax$, denoted by  $A^{\{0\}}x$, such that $\|A^{\{0\}}x\| = \inf\{\|x^*\|: x^*\in Ax\}.$  
	Let $\{\la_n\}\subset (0, \infty)$ be such that $\la_n\to 0 $ and $A_{\la_n}^\vp x\rh y$ in $ X^*$ as $n\to\infty$.  As in the proof of (iv), with $x_n = x$, we have $y\in Ax.$  In view of part (ii), it follows that
	$$\|y\| \le \liminf_{n\to\infty}\|A_{\la_n}^\vp x\| \le \limsup_{n\to\infty}\|A_{\la_n}^\vp x\|\le  \|A^{\{0\}}x\|, $$  and therefore we must have $y = A^{\{0\}}x$  and  $A_{\la_n}^\vp x\rh A^{\{0\}}x$ in $ X^*.$  Moreover, if $X^*$ is uniformly convex, then, by  \cite[Lemma~1.1]{BA}, we get
	$A_{\la_n}^\vp x\to A^{\{0\}}x$ in $ X^*.$ \\
	\vspace{-2ex}
	
	(vi) Suppose, on the contrary,  that there is  a sequence $\{\lambda_n\}$ with $\lambda_n\to 0$ and an element $x\not\in\overline{D(A)}$ such that $\{\|A_{\lambda_n}^\vp x\|\}$ is bounded. Let $R>0$ be such that $\|A_{\lambda_n}^\vp x\|\le R$ for all $n$.  Then, by (\ref{Yosida}), we have 
	$$\vp(\|x- J_{\lambda_n}^\vp x\|) =\|J_\vp(x- J_{\lambda_n}^\vp x)\|\le R\lambda_n .$$ Since  $\vp^{-1}$ is also a gauge function, we get
	$J_{\lambda_n}^\vp x \to x\mbox{  as } n\to\infty.$ This implies $x\in \overline{D(A)},$ a contradiction. 
\end{proof}

A proof of the following lemma  for $\vp(r) = r$  can be found in Boubakari and Kartsatos \cite{BK}.  Since we are dealing here with an arbitrary gauge function $\vp,$  we  provide  a complete proof. 

\begin{lemma}\label{L3}
	Let $A:X\supset D(A)\to 2^{X^*}$ be  maximal monotone  and $G\subset X$ be  bounded. Let  $0<\lambda_1<\lambda_2$. Then there exists  a constant $K$, independent of $\lambda$,  such that 
	$$\|A_\lambda^\vp x\| \le K$$
	for all $x\in \overline{ G}$ and $\lambda\in [\lambda_1, \lambda_2].$
\end{lemma}

 \begin{proof} For every $x\in X$, we have
$$	A_\lambda^\vp x  = \frac{1}{\lambda}J_\vp(x- x_\la), $$
where $x_\la = J_\lambda^\vp x.$ Let $(u, v) \in {\rm Gr}(A)$.   In view of (\ref{23}) in the proof of (i) in  Proposition~\ref{prop0}, we have
\begin{eqnarray*}
	\vp(\|x_\la-x\|)\|x_\la-x\| &\le& \la\|v\| \left(\|u-x\|+\|x_\la-x\|\right) +\vp(\|x_\la-x\|)\|u-x\|\\
	&\le& \la_2\|v\| \left(\|u-x\|+\|x_\la-x\|\right) +\vp(\|x_\la-x\|)\|u-x\|.
\end{eqnarray*}
By the properties of the gauge function $\vp$, it follows that
$\vp(\|x_\la -x\|) $ must be bounded, i.e., there exists a constant $K_0>0$ such that 
$$\vp(\|x_\la -x\|) \le  K_0 $$ for all $x\in \overline{ G}$ and all $\lambda \in [\lambda_1, \lambda_2].$  Consequently, we have
$$\|A_\lambda^\vp x \| = \frac{1}{\lambda}\vp(\|x_\la -x\|) \le \frac{1}{\lambda_1} K_0 =:K$$
for all $x\in \overline{ G}$ and all $\lambda \in [\lambda_1, \lambda_2].$
\end{proof}

By a well-known renorming theorem due to Troyanski \cite{Troyanski},
 a reflexive Banach space $X$ can be renormed with an equivalent
norm with respect to which both $X$ and $X^*$ become locally uniformly convex
(therefore strictly convex). With such a renorming, the  duality mapping $J_\varphi $ is a homeomorphism  from $X$ onto $X^*.$
Henceforth, we assume that  both $X$ and  $X^*$ are reflexive and locally uniformly convex.

The following lemma involving  $A_\lambda^\vp$ and $J_\lambda^\vp$  plays an important role in the sequel. Its proof is omitted here because of its similarity to Lemma~1  in \cite{AK},  except that, for the general $\vp$ here, we  must make use of $$x = J_{\lambda}^\vp x+ J_\vp^{-1} (\la A_{\la}^\vp x)  \mbox{ and } \lan A_\lambda^\vp x,   J_\vp^{-1}(\la A_{\lambda}^\vp x)\ran =  \vp^{-1}(\la \|A_\lambda^\vp x\|) \|A_\lambda^\vp x\|, \quad x\in X.$$
The lemma for  $A_\lambda$ and $J_\lambda$ is essentially due to Br\'ezis et al. \cite{BCP}.

\begin{lemma}\label{L1}
Let
	$A:X\supset D(A)\to 2^{X^*}$ and $S:X\supset D(S)\to 2^{X^*}$
	be maximal monotone operators such that $0\in D(A)\cap D(S)$ and
	$0\in S(0)\cap A(0)$. Assume
	that $A+S$ is maximal monotone and
	that there is a sequence $\{\lambda_n\}\subset
	(0,\infty)$ such that $\lambda_n\to 0$, and a sequence
	$\{x_n\}\subset D(S)$ such that $x_n\rightharpoonup x_0\in X$ and
	$A^\vp_{\lambda_n}x_n+w^*_n\rightharpoonup y_0^*\in X^*$, where
	$w^*_n\in Sx_n$. Then the following  statements are true.
	\begin{enumerate}
		\item The inequality
		\begin{equation} \label{L12}
			\lim_{n\to\infty}\langle A^\vp_{\lambda_n}x_n+w^*_n,x_n-x_0\rangle < 0
		\end{equation}
		is impossible.
		\item If
		\begin{equation} \label{L13}
			\lim_{n\to\infty}\langle A^\vp_{\lambda_n}x_n+w^*_n,x_n-x_0\rangle = 0,
		\end{equation}
		then $x_0\in D(A+S)$ and $y_0^*\in (A+S)x_0$.
	\end{enumerate}
\end{lemma}

\begin{definition}\label{D3} \rm
	An operator $A:X \supset D(A) \to 2^{X^*}$ is said to be ``\textit{strongly
		quasibounded}" if for every $S>0$ there exists $K(S)>0$ such that
	$\|x\| \le S$  and $\langle x^*, x \rangle \le S$
 for some  $x^*\in Ax	$
	imply $\|x^*\| \le K(S)$.
\end{definition}
It is obvious that a bounded operator is strongly quasibounded.  With regard to possibly unbounded operators, Browder and Hess \cite{BrowderHess1972}  and Pascali and Sburlan \cite{PS} have shown  that a monotone operator $A$
with $0\in{\rm Int} D(A)$ is strongly quasibounded.  The  following lemma with the particular case  $\vp(r) = r$ addressed in  Kartsatos and Quarcoo\cite[Lemma D]{Kartsatos2008} is needed in the sequel.

\begin{lemma}\label{L2}
	Let $A:X \supset D(A) \to 2^{X^*}$ be a strongly quasibounded
	maximal monotone operator such that $0 \in A(0)$. Let $\{\lambda_n\}
	\subset (0, \infty)$ and $\{x_n\} \subset X$ be such that
	$$
	\|x_n\| \le S \mbox{ and } \langle A_{\lambda_n}^\vp x_n, x_n \rangle \le
	S_1\quad\text{for all } n,
	$$
	where $S, S_1$ are positive constants. Then there exists a number
	$K>0$ such that $\|A_{\lambda_n}^\vp x_n\| \le K$ for all $n$.
\end{lemma}
\begin{proof}
Denote $w_n= A_{\lambda_n}^\vp x_n \mbox{ and } u_n =J_{\la_n}^\vp x_n $ for all $n.$  Then we have
$$w_n \in A u_n \mbox{ and }x_n = u_n+  J_\vp^{-1} (\la_nw_n).$$
In view of $0\in A(0)$, we obtain
\begin{eqnarray*}
	0\le \lan w_n, u_n \ran  &=& \lan w_n , x_n -  J_\vp^{-1}(\la_nw_n)\ran\\
	&=&\lan w_n , x_n \ran -\lan w_n, J_\vp^{-1}(\la_n w_n)\ran\\
	&=&\lan w_n , x_n \ran -\vp^{-1}(\la_n\| w_n\|) \|w_n\|\\
	&\le&S_1-\vp^{-1}(\la_n\| w_n\|) \|w_n\|.
\end{eqnarray*}
This yields $\lan w_n, u_n\ran \le S_1$ and $\vp^{-1}(\la_n\| w_n\|) \|w_n\|\le S_1$ for all $n$.  Suppose $\{w_n\}$ is unbounded. Then there exists a subsequence, denoted again by $\{w_n\}$, such that  $\|w_n\| \to \infty$ and $1\le \|w_n\|$ for all $n.$   Consequently, 
$\vp^{-1}(\la_n\| w_n\|)\le S_1$ for all $n$, and since $x_n  = u_n +J_\vp^{-1} (\lambda_n w_n)$, it follows that $$ \lambda_n \|w_n\| = \|J_\vp (x_n-u_n)\|= \vp(\|x_n -u_n\|) .$$ This implies
$\|x_n -u_n\| = \vp^{-1}(\la_n\| w_n\|)\le S_1$ for all $n.$  Since $\{x_n\}$ is bounded, we obtain the boundedness of $\{u_n\}$  and $\{\lan w_n, u_n\ran \},$  which contradicts the strong quasiboundedness of $A.$   Consequently,  $\{w_n\}$ is bounded.
\end{proof}

For the rest of this paper,  we take the gauge function $\varphi(r) = r^{p-1}, \;p>1.$    For the special case  $\vp(r) = r,$  the reader can find   proofs of Lemma~\ref{L4}  in Kartsatos and Skrypnik \cite{KartsatosSkrypnik2005a}  when $0\in A(0)$ and in  Asfaw and Kartsatos \cite{ASK2012},   without the condition $0\in A(0)$.  We note that Zhang and Chen in \cite[Lemma~2.7]{ZC} proved the continuity of $x\mapsto A_\lambda x$   on $ D(A)$ for each $\lambda>0$, also  without the condition $0\in A(0)$.   In  \cite[Lemma~6]{ASK2012}, however,   the continuity of $x\mapsto A_\lambda x$ on $X$ is used with no mention of its validity.   We next provide  a detailed proof of the  continuity of the mapping $(\lambda, x)\mapsto A_\lambda^\vp x$ on $(0, \infty) \times X.$

\begin{lemma}\label{L4}
	Let $A:X\supset D(A)\to 2^{X^*}$ be  a maximal monotone operator.  Then the mapping $(\lambda, x)\mapsto A_\lambda^\vp x$ is continuous on $(0, \infty) \times X$. 
\end{lemma}
\begin{proof}
We first prove the continuity of  $x\mapsto A_{\lambda_0}^\vp x$ on $X$ for each fixed $\lambda_0>0$. To this end, let $\{x_n\}\subset X$ be such that $x_n\to x_0\in X$.   
By Lemma~\ref{L3},  we have the boundedness of  $\{A_{\lambda_0}^\vp x_n\},$ and therefore
\begin{equation}\label{100}
	\lim\limits_{n\to\infty}\lan A_{\lambda_0}^\vp x_n- A_{\lambda_0}^\vp x_0, x_n-x_0\ran = 0.
\end{equation}
We know that 
\begin{equation}\label{101}x_n =J_{\lambda_0}^\vp x_n +\lambda_0^{q-1} J_\vp^{-1}(A_{\lambda_0}^\vp x_n)  \mbox{ and } x_0 =J_{\lambda_0}^\vp x_0 +\lambda_0^{q-1} J_\vp^{-1}(A_{\lambda_0}^\vp x_0).\end{equation}
Since $A_{\lambda_0}^\vp x_n \in A(J_{\lambda_0}^\vp x_n)$ and $A_{\lambda_0}^\vp x_0 \in A(J_{\lambda_0}^\vp x_0)$, the monotonicity of $A$ together with  (\ref{100}) and (\ref{101}) yields
\begin{equation}\label{102}
	\lim\limits_{n\to\infty}\lan A_{\lambda_0}^\vp x_n- A_{\lambda_0}^\vp x_0, J_\vp^{-1}(A_{\lambda_0}^\vp x_n)- J_\vp^{-1}(A_{\lambda_0}^\vp x_0)\ran = 0.
\end{equation}
Since $J_\vp^{-1}$ is a duality mapping from $X^*$ to $X$,     it follows,  in view of \cite[ Proposition~2.17]{Cioranescu},  that
$$A_{\lambda_0}^\vp x_n\to A_{\lambda_0}^\vp x_0 \mbox{ as }n\to\infty.$$ This proves the continuity  of $A_{\lambda_0}^\vp$ on $X$.

We now proceed to prove the continuity of $(\lambda, x)\mapsto A_\lambda^\vp x$  on $(0, \infty) \times X$. Let $\{\lambda_n\}\subset (0, \infty)$ and $\{x_n\}\subset X$ be such that $\lambda_n\to \lambda_0\in (0, \infty)$ and $x_n\to x_0\in X$ as $n\to\infty$.
Let $G\subset X$ be a bounded set that contains $x_n$ for all $n$.  Rename $\lambda_1, \lambda_2>0$ such that $\lambda_n\in [\lambda_1, \lambda_2]$ for all $n$. Since
$$ J_{\lambda_n}^\vp x_n \in A^{-1}(A_{\lambda_n}^\vp x_n) \mbox{ and }  x_n =J_{\lambda_n}^\vp x_n +\lambda_n^{q-1} J_\vp^{-1}(A_{\lambda_n}^\vp x_n),$$
it follows that
$$\begin{aligned}
	J_{\lambda_n}^\vp x_n + \lambda_0^{q-1} J_\vp^{-1}(A_{\lambda_n}^\vp x_n)&\in A^{-1}(A_{\lambda_n}^\vp x_n) +\lambda_0^{q-1} J_\vp^{-1}(A_{\lambda_n}^\vp x_n)\\& = \left(A^{-1} +\lambda_0^{q-1} J_\vp^{-1}\right)(A_{\lambda_n}^\vp x_n).
\end{aligned}$$ 
This implies
\begin{eqnarray*}
	A_{\lambda_n}^\vp x_n&=&\left(A^{-1}+\lambda_0^{q-1}J_\vp^{-1}\right)^{-1}\left(J_{\lambda_n}^\vp x_n + \lambda_0^{q-1} J_\vp^{-1}(A_{\lambda_n}^\vp x_n)\right)\\
	&=& A_{\lambda_0}^\vp\left(J_{\lambda_n}^\vp x_n + \lambda_0^{q-1} J_\vp^{-1}(A_{\lambda_n}^\vp x_n)\right)\\
	&=&A_{\lambda_0}^\vp\left(x_n + (\lambda_0^{q-1}- \lambda_n^{q-1}) J_\vp^{-1}(A_{\lambda_n}^\vp x_n)\right).
\end{eqnarray*}
By Lemma~\ref{L3},  $\{A_{\lambda_n}^\vp x_n\}$ is bounded, and so is  $\{J_\vp^{-1}(A_{\lambda_n}^\vp x_n)\}.$ Since $\lambda_n\to \lambda_0$, we have $ (\lambda_0^{q-1}- \lambda_n^{q-1}) J_\vp^{-1}(A_{\lambda_n}^\vp x_n) \to 0$ as $n\to\infty.$ The continuity of $A_{\lambda_0}^\vp$ implies 
$A_{\lambda_n}^\vp x_n \to 	 A_{\lambda_0}^\vp x_0$ as  $n\to \infty.$
This completes the proof.
\end{proof}

\begin{remark}\rm 
		We anticipate that Lemma~\ref{L4} holds for any gauge function $\vp.$  Since the  formula (\ref{closed-form})  may not hold for $A_\la^\vp $ with a general $\vp$,  the above proof does not go through and this subject may be of  independent research interest. 
\end{remark}

Let $G$ be an open and bounded subset of $X$. Let $L:X\supset D(L)\to X^*$
be densely defined linear maximal monotone, $A:X\supset D(A)\to 2^{X^*}$
maximal monotone, and $C(s) :X\supset \overline{G}\to X^*$, $s\in[0,1],$ 
a bounded homotopy of type $(S_+)$ w.r.t. $D(L)$.
Since  ${\rm Gr}(L)$ is closed in $X\times X^*,$ the space
$Y=D(L)$ associated with the graph norm
$\|x\|_Y = \|x\|_X + \|Lx\|_{X^*},$ $ x\in Y,$
becomes a real reflexive Banach space. We may assume that $Y$ and
its dual $Y^*$ are locally uniformly convex.

Let $j: Y\to X$ be the natural embedding and $j^* : X^*\to Y^*$ its
adjoint. Since $j:Y\to X$ is continuous, we have $D(j^*)=X^*.$
This implies that $j^*$ is also continuous. Since $j^{-1}$ is not necessarily
bounded, we have, in general, $j^*(X^*) \neq Y^*$.
Moreover, $j^{-1}(\overline G)=\overline G\cap D(L)$ is closed and $j^{-1}(G)=G\cap D(L)$
is open, 
$\overline{j^{-1}(G)} \subset j^{-1}(\overline G),$ and $ \partial (j^{-1}(G))\subset j^{-1}(\partial G).
$

We define $M:Y\to Y^*$ by
$
( Mx, y) := \langle Ly, J^{-1}(Lx)\rangle,$ $x, y\in Y,$ where
the duality pairing $(\cdot, \cdot )$ is in $Y^*\times Y,$ and
$J^{-1}$ is the inverse of the duality map $J:X\to X^*$ and is
identified with the duality map from $X^*$ to $X^{**}=X$. Also,
for every $x\in Y$ such that $Mx\in j^*(X^*)$, we have
$J^{-1}(Lx)\in D(L^*),$  $Mx = j^*\circ L^*\circ J^{-1}( Lx),\label{M2}$ and
\begin{gather*}
	( Mx-My,x-y) = \langle Lx-Ly,J^{-1}(Lx)-J^{-1}(Ly)\rangle \ge 0 \label{M2b}
\end{gather*}
for all $y\in Y$ such that $My\in j^*(X^*)$. Moreover, it is easy to see that $M$ is continuous on $Y$, and therefore $M$ is  maximal monotone.

We now define $\hat L: Y\to Y^*$ and $\hat C(s): j^{-1}(\overline G)\to Y^*$ by
$
\hat L = j^*\circ L \circ j$ and $\hat C(s) = j^*\circ C(s)\circ j,
$
respectively and, for each $t >0$, we also define $\hat A_t^\vp:Y\to Y^*$ by
$
\hat A_{t}^\vp = j^*\circ A_{t}^\vp\circ j,
$
where $A_t^\vp $ is the Yosida approximant of $A$ corresponding to the gauge function $\vp.$

The next lemma   employs   Lemma~\ref{L2}  and follows as in   \cite[Lemma~5]{AK2008}, and therefore its proof is omitted. 
\begin{lemma}\label{L30}
	Let $G\subset X$ be open and bounded. Assume the following:
	\begin{enumerate}
		\item $L:X\supset D(L)\to X^*$ is linear, maximal monotone with
		$\overline{D(L)}=X;$
		\item $A:X\supset D(A)\to 2^{X^*}$ is strongly quasibounded, maximal
		monotone with  $0\in A(0)$; and
		\item $C(t):X\supset\overline G\to X^*$ is a bounded homotopy of type
		$(S_+)$ w.r.t. $D(L)$.
	\end{enumerate}
	Then, for a continuous curve $f(s), 0\le s \le 1$, in $X^*$, the set
	$$
	F=\big\{x\in j^{-1}(\overline G): \hat L + \hat A_t^\vp +\hat C(s)+tMx = j^*f(s)
	\text{ for some }t>0,\; s\in[0,1]\big\}
	$$
	is bounded in $Y$. 
\end{lemma}

The next two propositions are essential for  the existence results in Section~\ref{S2} and Section~\ref{S3}.

\begin{proposition}\label{Prop1}
	Let $A:X\supset D(A)\to 2^{X^*}$ be maximal monotone and $C:X\supset D(C)\to X^*$ be  bounded, demicontinuous and  of type $(S_+)$. Suppose that $G\subset X$ is open and bounded such that $0\in A(0)$,  $p\in X^*,$ and   $$p\not\in (A+C)x$$ for all $x\in \partial G\cap D(A)\cap D(C).$  Then  the following statements hold.
	\begin{enumerate}
		\item There exists $t_0>0$ such that 
		$$A_t^\vp x+Cx\ne p$$ for all $x\in \partial G\cap  D(C)$ and  $t<t_0$.
		\item For  fixed $t_1, t_2>0$,  define $q(t) := t t_1+(1-t)t_2, \;t\in [0, 1]$. Then the operator $$H(t, x) = A_{q(t)}^\vp x+Cx, \quad (t, x) \in [0, 1]\times \overline{G}$$ is a homotopy of type $(S_+)$. 
		\item For every sequence $\{t_n\}\subset(0, \infty)$ such that $t_n\to 0,$   $\lim\limits_{n\to\infty}{\rm d}_{S_+} (A_{t_n}^\vp+C, G, p)$ exists and does not depend on the choice of $\{t_n\}$. 
	\end{enumerate}
\end{proposition}
\begin{proof}
	(i) Without loss of generality, we assume that $p = 0$. In fact, if $p\ne 0$, then we replace $C$ with $C-p.$  Suppose  that (iii) is false. Then there exist $\{t_n\}\subset (0, \infty)$   and $\{x_n\}\subset \partial G$ such that $t_n\to 0$ and 
	\begin{equation}\label{AC-equation}
		A_{t_n}^\vp x_n +Cx_n = 0
	\end{equation} for all $n$. Since $C$ is bounded, $\{Cx_n\}$ is bounded. This implies that $\{A_{t_n}^\vp x_n\}$ is also bounded.  We may assume that there exist $x_0\in X$ and $w_0\in X^*$ such that $x_n\rightharpoonup x_0$ in $X$ and $A_{t_n}^\vp x_n\rightharpoonup w_0$ in $X^*.$    If  \[\limsup_{n\to\infty} \lan Cx_n, x_n-x_0\ran >  0,\] we find a subsequence of $\{x_n\},$  denoted again by itself, such that 
	\[\lim_{n\to\infty} \lan Cx_n, x_n-x_0\ran >  0.\]   In view of (\ref{AC-equation}), we obtain
	
	$$\lim_{n\to\infty} \lan A_{t_n}^\vp x_n, x_n-x_0\ran <  0;$$  however, this  is impossible by  (i)  of Lemma~\ref{L1}.   We then must have
	\[\limsup_{n\to\infty} \lan Cx_n, x_n-x_0\ran \le  0.\]  By the $(S_+)-$property of $C$, we have $x_n\to x_0,$ and consequently
	$$\lim_{n\to\infty} \lan A_{t_n}^\vp x_n, x_n-x_0\ran =  0.$$  By (ii) of Lemma~\ref{L1},  we get $x_0\in D(A)$ and  $w_0\in Ax_0.$
 Since $C$ is demicontinuous,  $Cx_n \rh Cx_0$ in $X^*.$  This implies 
	$w_0= - Cx_0, $ i.e., $0\in (A+C)(\partial G)$, contradicting  $0\notin (A+C)(\partial G).$\\
	\vspace{-2ex}
	
	(ii) Let $t_1, t_2\in (0, t_0]$ be such that $t_1< t_2$.  Consider the following one-parameter family of operators:
	$$H(t, x) := A_{q(t)}^\vp  x+ Cx, \quad (t, x) \in [0, 1]\times \overline{G}.$$
	We prove that $H(t, \cdot)$ is a bounded  homotopy of type $(S_+)$.   The boundedness of $H(\cdot, \cdot)$ follows from Lemma~\ref{L3} and the boundedness of $C$. Let $\{t_n\}\subset [0, 1]$ and $\{x_n\}\subset \overline{ G}$ satisfy $t_n \to t_0$ and $x_n\rh x_0$ in $X,$ and 
	\begin{equation}\label{201}
		\limsup_{n\to\infty}\lan A_{q(t_n)}^\vp  x_n+ Cx_n, x_n -x_0\ran \le 0.
	\end{equation} Using the monotonicity  of $A_{q(t)}^\vp$ in (\ref{201}), we get
	\begin{equation}\label{202}
		\limsup_{n\to\infty}\lan A_{q(t_n)}^\vp  x_0+ Cx_n, x_n -x_0\ran \le 0.
	\end{equation} By Lemma~\ref{L4}, we have $A_{q(t_n)}^\vp  x_0 \to A_{q(t_0)}^\vp  x_0,$  and therefore (\ref{202}) yields
	$$\limsup_{n\to\infty}\lan  Cx_n, x_n -x_0\ran \le 0.$$ Since $C$ is  demicontinuous and of type $S_+$, it follows that $x_n \to x_0$ in $X$ and $Cx_n \rightharpoonup Cx_0$ in $X^*.$ Consequently, we have $$A_{q(t_n)}^\vp  x_n  +Cx_n \rh  A_{q(t_0)}^\vp  x_0+Cx_0$$ as $n\to \infty$. This proves that $H(t, \cdot)$, $t\in [0, 1]$, is a homotopy of type $(S_+)$.  \\
	\vspace{-2ex}
	
	(iii) By the invariance of the degree, ${\rm d}_{S_+},$  for $(S_+)$-mappings under the homotopies of type $(S_+)$, we have
	$${\rm d}_{S_+} (A_{t_1}^\vp, G, 0)= {\rm d}_{S_+} (H(0, \cdot), G, 0)= {\rm d}_{S_+}( H(1, \cdot), G, 0)={\rm d}_{S_+} (A_{t_2}^\vp, G, 0).$$
	It follows that ${\rm d}_{S_+} (A_t^\vp, G, 0)$ exists and is independent of $t\in (0, t_0]$.
\end{proof}

\begin{remark}\rm  Let $A$, $C$, $G,$ and $p$ be the same as in Proposition~\ref{Prop1}.
	When we define a degree mapping of $A+C$, denoted by ${\rm D} (A+C, G, p)$,   by 
	$$  {\rm D} (A+C, G, p) = \lim_{t\to 0^+} {\rm d}_{S_+} (A_t^\vp, G, p),$$ we can verify that the degree mapping $ {\rm D}$ has the same four basic properties as the  Browder degree in \cite{Browder1983}.   By the uniqueness of the Browder degree  established by  Berkovits and Miettunen \cite{BM2008}, the degree ${\rm D}$  coincides with the Browder degree  for $A+C$.
\end{remark}

By  replacing $\hat T_t$ everywhere in  \cite[Lemma 5, Lemma 6, and Lemma 8]{AK2008} with   $\hat A_t^\vp$  with the gauge function $\vp(r) = r^{p-1}$ and by following the methodology used in \cite{AK2008} in conjunction with  Lemma~\ref{L1},  Lemma~\ref{L2}, Lemma~\ref{L4},  and Lemma~\ref{L30}, we obtain Proposition~\ref{Prop2}  below. Its proof is omitted here because the method of proof is similar to that  in \cite{AK2008} and Proposition~\ref{Prop1}, except for having to deal with $\hat A_t^\vp.$  For further properties of $L+A+C$ in relation to the following proposition   for  $\vp(r) = r$,  the reader is referred to Addou and Mermri~\cite{AM} and Adhikari and Kartsatos\cite{AK2008}.
\begin{proposition}\label{Prop2} Let $G\subset X$ be  open and bounded. Assume that
	$L:X\supset D(L)\to X^*$ is linear, maximal monotone with
	$\overline{D(L)}=X$; $A:X\supset D(A)\to 2^{X^*}$ is strongly quasibounded, maximal
	monotone with  $0\in A(0)$; and $C(t):X\supset\overline G\to X^*$, \; $t\in [0, 1]$,  is a bounded homotopy of type
	$(S_+)$ w.r.t. $D(L)$.
	Suppose that  $$0\not\in (L+A+C(t))x$$ for all $x\in \partial G\cap D(L)\cap D(A).$  Then  the following statements hold.
	\begin{enumerate}
		\item There exists $t_0>0$ such that 
		$$\hat Lx + \hat A_t^\vp x+\hat C(t)x+t Mx\ne 0$$ for all $(t,x)\in [0, 1]\times (\partial G\cap  D(L))$ and  $t<t_0$.
		\item For  fixed numbers $t_1, t_2>0$,  define $q(t) := tt_1+(1-t)t_2,\;  t\in [0, 1]$. Then  the operator \[\hat H(t, x) =  \hat L x+ \hat A_{q(t)}^\vp x +\hat C(t)x+s(t) Mx,\]  $(t,x)\in [0, 1]\times (\overline G\cap  D(L)),$ is a homotopy of type $(S_+)$. 
		\item For every sequence $\{t_n\}\subset (0, \infty)$ such that $t_n\to 0,$  $$\lim\limits_{n\to\infty}{\rm d}_{S_+} (\hat L + \hat A_{t_n}^\vp +\hat C(t)+t_n M, G, 0)$$ exists and does not depend on the choice of $\{t_n\}$. 
	\end{enumerate}
\end{proposition}

\section{Existence of Nontrivial Solutions}\label{S3}

 Hu and Papageorgiou  generalized in    \cite{HP} the 
Browder  degree theory \cite{Browder1983} to the mappings of the form $A+C+T$, where
$A:X\supset D(A)\to 2^{X^*}$ is maximal monotone with $0\in A(0)$, $C:X\supset D(C)\to X^*$ is bounded demicontinuous of type
$(S_+),$ and $T$ is of   class $(P)$. With an application
of the $(S_+)$-degree developed by Browder \cite{Browder1983} and Skrypnik \cite{Skrypnik1994},   we prove in Theorem~\ref{Th1}  the existence of nonzero solutions of  $Ax+Cx+Tx\ni 0$  when  $A+C+T$ satisfies certain  boundary conditions, and  the operator $A,$ in addition, is positively homogeneous of degree $\gamma>0$.
This result extends  the  existence result for $\gamma \in (0, 1]$ in \cite{Adhikari2017}    to $\gamma >0$ (see also \cite[Theorem 6]{AK} for $\gamma =1$).   

The following lemma, which  is crucial to the existence results in this section, shows that positively homogeneous maximal monotone operators transmit the  homogeneity into  their Yosida approximants corresponding to  $J_\vp$ with $\vp(r) = r^{p-1}$, $p>1$,   and a suitable   value of $ p.$

\begin{lemma}\label{L5}
	Let $A:X\supset D(A)\to 2^{X^*}$ be maximal monotone and positively homogeneous
	of degree  $\gamma>0$. Then, for each $t>0$, the Yosida approximant $A_t^\vp$ corresponding to the gauge function $\vp(r) = r^{p-1}$, $p>1$,  satisfies
	\begin{equation}\label{138}
		A_t^\vp (sx) =\begin{cases} 
			s^\gamma A_{ts^{\gamma+1-p}}^\vp x & \mbox{for  }
			(s, x)\in (\mathbf{R}_+\setminus\{0\})\times X\\
			0 & \mbox{for  }
			(s, x)\in \{0\}\times X.
		\end{cases} 
	\end{equation}
	Consequently, if $p= \gamma+1, $   then $A_t^\vp$ is positively homogeneous of degree $\gamma$, i.e.,  $$A_t^\vp (sx) = s^\gamma A_t^\vp x \quad \mbox{for all }
	(s, x)\in \mathbf{R}_+\times X.$$
\end{lemma}

\begin{proof}
Let $t>0$ be fixed. The case $s=0$ is trivial. 	Assume $s>0$, and let
$$
y= A_t^\vp (sx) =(A^{-1} + t ^{q-1}J_\vp ^{-1}) ^{-1}(sx), \quad x\in X,
$$
where $q$ satisfies $1/p+1/q =1.$
Then
\begin{equation*}
	y \in A(-t ^{q-1}J_\vp^{-1} y + sx)\\
	= A\left(s\left(-t ^{q-1}s^{-1} J^{-1}_\vp  y +  x\right)\right).
\end{equation*}
This means $$\left(s\left(-t ^{q-1}s^{-1} J^{-1}_\vp  y +  x\right), \; y\right)\in {\rm Gr} (A).$$ Since $A$ is positively homogeneous of degree $\gamma>0$, we get
$$\left(-t ^{q-1}s^{-1} J^{-1}_\vp  y +  x, \; s^{-\gamma}y\right)\in {\rm Gr} (A),$$ 
i.e.,  $$s^{-\gamma}y\in A\left(-t ^{q-1}s^{-1} J^{-1}_\vp  y +  x\right).$$
In view of  (\ref{11}), we have  $$s^{-\gamma(1-q)}J^{-1}_\vp (s^{-\gamma} y)=J^{-1}_\vp y,$$  and therefore
$$s^{-\gamma} y\in A\left(-t ^{q-1}s^{\gamma (q-1)-1} J^{-1}_\vp (s^{-\gamma} y) +
x\right).$$ This implies
$$x\in \left(A^{-1} +t ^{q-1}s^{\gamma (q-1)-1} J^{-1}_\vp \right)(s^{-\gamma} y).$$
Using $$t ^{q-1}s^{\gamma (q-1)-1}= \left(ts^\gamma\right)^{q-1}\left(s^{1-p}\right)^{q-1} = \left(ts^{\gamma+1-p}\right)^{q-1},$$
we get
$$
y = s^\gamma\left(A^{-1} +\left(ts^{\gamma+1-p}\right)^{q-1}  J^{-1}_\vp \right)^{-1}x
= s^\gamma A_{ts^{\gamma+1-p}}^\vp x.
$$
Thus, we have
$$A_t^\vp (sx) = s^\gamma A_{ts^{\gamma+1-p}}^\vp x.$$  Clearly, $A_t^\vp$ is positively homogeneous of degree $\gamma$  if $p=\gamma +1 .$
\end{proof} 

\begin{remark}\rm 
	In the settings of Lemma~\ref{L5} with $p=\gamma +1$, it follows from (\ref{splitting}) that the  resolvent $J_t^\vp$ is positively homogeneous of degree 1 in the following sense:  for each $t>0$, we have $J_t^\vp(sx) = s J_t^\vp x$ for all $x\in X$ and all $s\ge 0.$
\end{remark}

\begin{theorem}\label{Th1}
	Assume that $G_1, G_2\subset X$ are open, bounded with $0\in G_2$
	and $\overline{G_2}\subset G_1$. Let $A:X\supset D(A)\to  2^{X^*}$
	be maximal monotone and positively homogeneous of degree
	$\gamma>0$ with $A(0) = \{0\};$ $C:\overline{ G}_1\to X^*$ bounded,
	demicontinuous and of type $(S_+);$ and $T:\overline{G}_1\to
	2^{X^*}$ of class $(P)$. Assume, further, that
	\begin{enumerate}
		\item[\rm (H1)] there exists $v^*_0\in X^*\setminus\{0\}$ such that
		$Ax+Cx +Tx\not\ni \lambda v^*_0$  for all $(\lambda,x)\in\mathbf{R}_+\times(D(A)\cap\partial G_1),$  and
		\item[\rm (H2)] $Ax+Cx+ Tx+\lambda Jx \not\ni 0$ for all $(\lambda,x)\in
		\mathbf{R}_+\times(D(A)\cap\partial G_2)$.
	\end{enumerate}
	Then the inclusion $Ax+Cx+Tx\ni 0$ has a nonzero solution $x\in
	D(A)\cap(G_1\setminus G_2)$.
\end{theorem}

\begin{proof}  
	To study the solvability of the inclusion
	$$Ax + Cx + Tx \ni 0,  \; x\in \overline G_1,$$ we consider the associated approximate equation
	\begin{equation}\label{5}
		A_t^\vp x + Cx + q_\epsilon x = 0, \; t>0, x\in \overline G_1, \epsilon>0.
	\end{equation}
	Here, the gauge function is taken to be $\vp(r) = r^{p-1},$ $1<p<\infty$ so that $\gamma = p-1$,     
	and $q_\epsilon:\overline{G_1}\to X^*$  is an approximate
	continuous Cellina-selection as in \cite[Lemma 6]{Aubin1984} and \cite{HP}
	satisfying
	$
	q_\epsilon x\in T(B_\epsilon(x)\cap \overline{G_1})+B_\epsilon(0)
	$
	for all $x\in \overline{G_1}$ and
	$q_\epsilon(\overline{G_1})\subset\overline{\operatorname{co}
		T(\overline{G_1})}$.
	
	We show that the equation \eqref{5} has a solution $x_{t, \epsilon}$ in
	$G_1\setminus G_2$ for all sufficiently small $t$ and $\epsilon$.
	To this end, we first show that there exist $\tau_0>0$, $t_0>0$ and
	$\epsilon_0>0$ such that the equation
	\begin{equation}\label{6*}
		A_t^\vp x + Cx+ q_\epsilon x = \tau v_0^*
	\end{equation}
	has no solution in $G_1$ for every $\tau\ge \tau_0$, $t\in(0, t_0]$
	and $\epsilon\in(0, \epsilon_0]$.
	
	Assuming the contrary, let $\{\tau_n\}\subset(0, \infty)$,
	$\{t_n\}\subset(0, \infty)$, $\{\epsilon_n\}\subset(0,\infty)$ and
	$\{x_n\}\subset G_1$ be such that $\tau_n\to\infty$, $t_n\to 0$,
	$\epsilon_n\to 0$ and
	\begin{equation}\label{7}
		A_{t_n}^\vp x_n + Cx_n +q_{\epsilon_n}x_n = \tau_n v_0^*.
	\end{equation}
	We may assume that $q_{\epsilon_n}x_n\to g^*\in X^*$ in view of
	the properties of $T$. Then $\|A_{t_n}^\vp x_n\|\to \infty$ as
	$\|\tau_n v_0^*\|\to\infty$ and $\{Cx_n\}$ is bounded.
	Thus, from \eqref{7}, we obtain
	\begin{equation}\label{8}
		\frac{A_{t_n}^\vp x_n}{\|A_{t_n}^\vp x_n\|} +
		\frac{Cx_n}{\|A_{t_n}^\vp x_n\|} + \frac{q_{\epsilon_n}x_n}{\|A_{t_n}^\vp x_n\|}=
		\frac{\tau_n}{\|A_{t_n}^\vp x_n\|}v_0^*.
	\end{equation}
	This implies	
	\begin{equation}\label{10}
		\frac{\tau_n\|v_0^*\|}{\|A_{t_n}^\vp x_n\|}\to 1
		\quad\text{ so that} \quad
		\frac{\tau_n}{\|A_{t_n}^\vp x_n\|}\to\frac{1}{\|v_0^*\|} \quad\text{as } n\to\infty.
	\end{equation}
	Since $p-1 =\gamma$,  by Lemma~\ref{L5},  $A_t^\vp$ is also homogeneous of degree $\gamma = p-1$, and therefore we obtain
	\begin{equation} \label{144}
		\frac{A_{t_n}^\vp x_n}{\|A_{t_n}^\vp x_n\|}
		= A_{t_n}^\vp \left(\frac{x_n}{\|A_{t_n}^\vp x_n\|^{1/\gamma}}\right).
	\end{equation}
	Let
	$	u_n = {x_n}/{\|A_{t_n}^\vp x_n\|^{1/\gamma}}.	$
	It is clear that 
	$u_n\to 0$.  In view of \eqref{8}, \eqref{10}, and \eqref{144}, we obtain
	$A_{t_n}^\vp u_n\to h$ with
	$	h = {v_0^*}/{\|v_0^*\|}.$
	This implies
	$$	\lim_{n\to\infty}\langle A_{t_n}^\vp u_n, u_n\rangle = \langle h,
	0\rangle = 0.	$$
	Since $t_n\to 0$, by (ii) of
	Lemma~\ref{L1} with $S=0$ we obtain $0\in D(A)$ and $h\in A(0)=\{0\}$,
	a contradiction to  $\|h\| = 1$.
	
	We now consider the homotopy mapping
	\begin{equation}\label{13}
		H_1(s,x, t, \epsilon) = A_t^\vp x+Cx+q_\epsilon x - s\tau_0v_0^*,
		\quad s\in[0,1], \;x\in\overline{G_1},
	\end{equation}
	where $t\in(0, t_0]$ and $\epsilon\in(0, \epsilon_0]$ are fixed.   By following the arguments as in \cite[Theorem 3.1]{Adhikari2017}, we can show that, for
	every $s\in[0,1]$ the operator $x\mapsto Cx- s\tau_0v_0^*$ is
	bounded, demicontinuous and of  type $(S_+)$ on $\overline{G_1},$ and that
	the equation $H_1(s, x, t, \epsilon) = 0$
	has no solution in $\partial G_1$ for all sufficiently small
	$t\in(0, t_0]$, $\epsilon\in(0, \epsilon_0]$
	and all $s\in[0,1]$.   In doing this, we need to use Lemma~\ref{L1}.  The details are omitted.
	
 It follows from Proposition~\ref{Prop1} that the mapping $H_1(s, x, t, \epsilon)$ is an
	admissible homotopy for the degree,  ${\rm d}_{S_+}$,  of  $(S_+)$-mappings, and 
	${\rm d}_{S_+}(H_1(s,\cdot, t, \epsilon), G_1, 0)$ is well-defined and is
	a constant for all $s\in[0,1]$ and for all $t\in(0,t_0]$,
	$\epsilon\in(0, \epsilon_0]$.

	Assume that
	$$
	{\rm d}_{ S_+}(H_1(1, \cdot,t_1,\epsilon_1), G_1, 0)\ne 0,
	$$
	for some sufficiently small $t_1\in(0, t_0]$ and
	$\epsilon_1\in(0, \epsilon_0]$. Then the equation
	$$
	A_{t_1}^\vp x +Cx +g_{\epsilon_1} x = \tau_0v_0^*
	$$
	has a solution in  $ G_1$. However, this contradicts our
	choice of the number $\tau_0$ in \eqref{6*}. Consequently,
	$$
	{\rm d}_{S_+}(A_t^\vp +C+q_\epsilon, G_1, 0)
	= {\rm d}_{S_+}(H_1(0, \cdot,t,\epsilon), G_1, 0)= 0, \quad t\in(0, t_0],\;
	\epsilon\in(0, \epsilon_0].
	$$
	
	We next consider the homotopy mapping
	\begin{equation}\label{18}
		H_2(s, x, t,\epsilon) = s(A_t^\vp x+Cx +q_\epsilon x)+(1-s)Jx, \quad (s,
		x)\in[0,1]\times\overline{G_2}.
	\end{equation}
	
	We claim that  there exist $t_1\in(0, t_0]$  and $\epsilon_1\in(0,
	\epsilon_0]$ such that  $H_2(s, x, t,\epsilon)= 0$ has
	no solution on $\partial G_2$ for any $s\in[0,1]$, any $t\in(0,
	t_1]$ and any $\epsilon\in(0, \epsilon_1]$. To  prove the claim, we assume the contrary and then  follow the argument used   in  \cite[Theorem 3.1]{Adhikari2017} along with the properties of $A_t^\vp$ established in Lemma~\ref{L1} to arrive at a contradiction to (H2).
		For the sake of convenience, we
	assume that $t_0$ and $\epsilon_0$ are sufficiently small so that we
	may take $t_1 = t_0$ and $\epsilon_1 = \epsilon_0$.
	
	 It follows from Proposition~\ref{Prop1} that  $H_2(s, x, t,\epsilon)$ is
	an admissible homotopy for the degree of  $(S_+)$-mappings and 
	${\rm d}_{\rm S_+}(H_2(s, \cdot, t, \epsilon), G_2, 0)$ is well-defined and
	constant for all $s\in[0,1]$, all $t\in(0, t_0]$ and all
	$\epsilon\in(0,\epsilon_0]$.
	By the invariance of the $(S_+)$-degree, for all $t\in(0, t_0]$ and
	$\epsilon\in(0,\epsilon_0]$, we have
	\begin{align*}
		{\rm d}_{ S_+}(H_2(1, \cdot, t, \epsilon), G_2, 0)
		&= {\rm d}_{ S_+}(A_t^\vp +C+q_\epsilon, G_2, 0)\\
		&= {\rm d}_{ S_+}(H_2(0, \cdot, t, \epsilon), G_2, 0)\\
		&= {\rm d}_{ S_+}(J, G_2, 0)
		= 1.
	\end{align*}
	Thus, for all $t\in(0, t_0]$, $\epsilon\in(0,\epsilon_0]$, we have
	$$
	{\rm d}_{ S_+}(A_t^\vp +C+q_\epsilon, G_1, 0)
	\ne {\rm d}_{ S_+}(A_t^\vp +C+q_\epsilon, G_2,0).
	$$
	Using the excision property of the $(S_+)$-degree, which is an easy
	consequence of its finite-dimensional approximations,  for every $t\in(0, t_0]$ and every
	$\epsilon\in(0, \epsilon_0]$, there exists a
	solution $x_{t, \epsilon}\in G_1\setminus G_2$ of
	$A_t^\vp x+Cx+q_\epsilon x = 0.$ Let $t_n\in(0, t_0]$ and
	$\epsilon_n\in(0, \epsilon_0]$ be such that $t_n\to 0$,
	$\epsilon_n\to 0$ and let $x_n\in G_1\setminus G_2$ be the
	corresponding solutions of $A_t^\vp x +Cx+q_\epsilon x = 0$, i.e.,
	$$
	A_{t_n}^\vp x_n +Cx_n+ q_{\epsilon_n}x_n = 0.
	$$
	We may assume that $x_n\rightharpoonup x_0$ in $X$ and
	$q_{\epsilon_n}x_n\to g^*\in X^*$. We observe that
	$$
	\langle A_{t_n}^\vp x_n , x_n -x_0\rangle
	= -\langle Cx_n + q_{\epsilon_n}x_n, x_n- x_0\rangle.
	$$
	If
	$$
	\limsup_{n\to\infty}\langle Cx_n+q_{\epsilon_n}x_n, x_n-
	x_0\rangle >0,
	$$
	then we obtain a contradiction from (i) of
	Lemma~\ref{L1} with $S=0$ there. Consequently,
	$$
	\limsup_{n\to\infty}\langle Cx_n+q_{\epsilon_n}x_n, x_n-
	x_0\rangle \le 0,
	$$
	and hence
	$$
	\limsup_{n\to\infty}\langle Cx_n, x_n- x_0\rangle \le0.
	$$
	By the $(S_+)$-property of $C$, we obtain
	$x_n\to x_0\in\overline{G_1\setminus G_2}$. Then $Cx_n\rightharpoonup Cx_0$
	and $A_{t_n}^\vp x_n \rightharpoonup -Cx_0-g^*$. Using this in (ii) of
	Lemma~\ref{L1} with $S=0$ there, we obtain $x_0\in D(A)$ and $-Cx_0-g^*\in Ax_0$. By
	a property of the selection $q_{\epsilon_n}x_n$
	as in Hu and Papageorgiou \cite{HP}, we have $g^*\in Tx_0,$
	and therefore $Ax_0+Cx_0+Tx_0\ni 0.$ 
	We also have
	$$
	x_0\in\overline{G_1\setminus G_2}
	= (G_1\setminus G_2)\cup\partial(G_1\setminus G_2)
	\subset(G_1\setminus G_2)\cup\partial G_1\cup\partial G_2.
	$$
	By (H1) and (H2), we have $x_0\notin \partial G_1\cup\partial G_2,$ and hence $x_0\in D(A)\cap(G_1\setminus G_2).$
\end{proof}

\begin{remark}\rm 
	We point out that the condition $A(0) = \{0\}$ on the homogeneous maximal monotone operator $A$ used in Theorem~\ref{Th1} is rather  mild  in  view of Rockafellar's result \cite{Rockafellar} which says that a monotone
	map is locally bounded at every point in the interior of its domain.
\end{remark}

The existence of nonzero
solutions of $Lx+Ax+Cx\ni 0$, where  the maximal monotone operator $A$ is strongly quasibounded and positively homogeneous of degree $\gamma =1,$ is  established  in \cite{Adhikari2017}.  In the following theorem, we extend this result  to an arbitrary degree $\gamma >0$ for the same combination of operators in the spirit of the Berkovits-Mustonen
theory in \cite{BM} and the theories developed in \cite{AK}.  We recall that the maximal monotone operator $A$ investigated in \cite{AK}  is strongly quasibounded.  However, by a result of Hess \cite{Hess}, a strongly quasibounded and positively homogeneous operator of degree $\gamma>0$ is necessarily bounded.  Therefore, in the following theorem, we assume that the maximal monotone operator $A$  is bounded.

\begin{theorem}\label{Th2}
	Assume that $G_1, G_2\subset X$ are open, bounded with $0\in G_2$
	and $\overline{G_2}\subset G_1$.
	Let $L: X\supset  D(L)\to X^*$ be linear maximal monotone with $\overline{D(L)} =X$,
	and $A:X\supset D(A) \to 2^{X^*}$ bounded, maximal monotone
	and positively homogeneous of degree $\gamma>0.$ Also, let $C:\overline{G_1}\to X^*$ be
	bounded, demicontinuous and of type $(S_+)$ w.r.t. $D(L)$.
	Moreover, assume  that
	\begin{enumerate}
		\item[\rm (H3)] there exists $v^*\in X^*\setminus\{0\}$ such that
		$Lx+Ax+Cx\not\ni \la v^*$ for all $(\la,x)\in\mathbf{R}_+\times( D(L)\cap D(A)\cap \partial G_1)$, and
		\item[\rm(H4)] $Lx+Ax+Cx+ \la Jx \not\ni 0$ for all $(\la,x)\in~\mathbf{R}_+\times( D(L)\cap D(A)\cap \partial G_2)$.
	\end{enumerate}
	Then the inclusion $Lx+Ax+Cx\ni 0$ has a solution $x\in
	D(L)\cap D(A)\cap(G_1\setminus G_2)$.
\end{theorem}

\begin{proof}
		We begin by observing that a positively homogeneous and bounded maximal monotone operator $A$  of degree $\gamma>0$ satisfies $0\in D(A)$ and $A(0) =\{0\}$. 
	To solve the inclusion
	\begin{equation}\label{T1}
		Lx+Ax+Cx\ni 0, \quad x\in\overline{G_1},
	\end{equation}
	let us consider the associated equation
	\begin{equation}\label{T2}
		\hat Lx+\hat A_t^\vp x+\hat Cx +t Mx=0, \quad t\in (0, \infty),\;
		x\in j^{-1}(\overline{G_1}).
	\end{equation}
	Here, the gauge function  is   $\vp(r) = r^{p-1},$  $1<p<\infty$, and $\gamma = p-1.$
	We can show as in \cite[Lemma~5]{AK2008} that there exists $R>0$ such that the open
	ball $B_Y(0, R) $ contains all the  solutions of \eqref{T2}. We recall  that $Y= D(L).$
	
	We shall prove that \eqref{T2} has a solution $x_t\in j^{-1}(G_1\setminus G_2)$
	for all sufficiently small $t >0$.
	We first claim that there exist $\tau_0>0$ , $t_0>0$ such that
	\begin{equation}\label{T3}
		\hat Lx+\hat A_t^\vp x+\hat Cx +t Mx=\tau j^*v^*
	\end{equation}
	has no solution in $ G^1_R(Y):=j^{-1}(G_1)\cap B_Y(0, R)$ for all
	$t \in (0, t_0]$ and all $\tau \in [\tau_0, \infty)$.
	Assume the contrary and let $\{\tau_n\}\subset (0, \infty)$,
	$\{t_n\}\subset (0, 1)$ and $\{x_n\} \subset G^1_R(Y)$ such that
	$\tau_n\to \infty$, $t_n\to 0$ and
	\begin{equation}\label{T4}
		\hat Lx_n+\hat A_{t_n}^\vp x_n+\hat Cx_n +t_n Mx_n=\tau_n j^*v^*.
	\end{equation}
	We note that $j^*$ is one-to-one because $j(Y) = Y,$  which is dense in $X$.
	This implies that $j^*v^*$ is nonzero, and therefore
	$\|\tau_n j^*v^*\|_{Y^*}\to+\infty$. Also, the sequence $\{x_n\}$
	is bounded in $Y$ and so we may assume that $x_n\rightharpoonup x_0$
	in $X$ and $Lx_n\rightharpoonup Lx_0$ in $X^*$. In particular, $\{Lx_n\}$
	is bounded in $X^*$. Since $Mx_n \in j^*(X^*)$, we have $J^{-1}(Lu)\in D(L^*)$ and
	$$
	Mx_n = j^* L^* J^{-1}(Lx_n).
	$$
	Since $j^*$, $L^*$, $J^{-1}$ are bounded, we obtain the boundedness of
	$\{Mx_n\}$. It is clear that $\hat Lx_n$ and $\hat Cx_n$ are bounded in $Y^*$, and
	therefore \eqref{T4} implies that $\| \hat A_{t_n}^\vp x_n\|_{Y^*} \to \infty$.	
	Since $A$ is positively homogeneous of degree $\gamma$, applying Lemma~\ref{L5} for $\gamma = p-1$ shows that each $A_{t_n}^\vp$ is also positively 	homogeneous of  $\gamma=p-1$.
	Consequently,
	\begin{equation}\label{T55}
		\frac{\hat A_{t_n}^\vp  x_n}{\|\hat A_{t_n}^\vp  x_n\|_{Y^*}} = \hat A_{t_n}^\vp\left(\frac{ x_n}{\|\hat A_{t_n}^\vp  x_n\|_{Y^*}^{1/\gamma}}\right)
	\end{equation} for all $n$. Define $\beta_n := 1/{\|\hat A_{t_n}^\vp  x_n\|_{Y^*}} \mbox{ and } \delta_n:= \beta_n^{1/\gamma}. $  Since $\| \hat A_{t_n}^\vp x_n\|_{Y^*} \to \infty$, it follows that $\beta_n x_n \to 0$  and $\delta_n x_n \to 0$ in $X$ as $n\to\infty$. 
	From \eqref{T4} and \eqref{T55}, we find
	\begin{equation}\label{T5}
		\hat L (\beta_n x_n)+\hat A_{t_n}^\vp (\delta_n x_n)+\beta_n\hat Cx_n+t_n\beta_n Mx_n
		=\tau_n\beta_n j^*v^*.
	\end{equation}
	Because $\|\hat A_{t_n}^\vp (\delta_n x_n)\|_{Y^*} = 1$ and the remaining terms on the left in  \eqref{T5} converge to $0$ in $X^*$ as $n\to\infty$, we get 
	$
	\tau_n\beta_n \to 1/{\|j^*v^*\|_{Y^*}},
	$
	and therefore
	$\hat A_{t_n}^\vp (\delta_n x_n )\to y_0,$
	where
	$y_0 = j^*v_*/{\|j^*v^*\|_{Y^*}}.	$
	Since $u_n:= \delta_n x_n\to 0$ as $n\to\infty$, we have
	$
	\lan \hat A_{t_n}^\vp u_n, u_n\ran 
	\to\lan y_0, 0\ran =0
	$ as $n\to\infty$. 
	By Lemma \ref{L1}, (ii), we have
	$
	y_0\in A(0) =\{0\},
	$
	which is a contradiction to $\|y_0\|_{Y^*} = 1$.

	We now consider the homotopy $H: [0,1]\times Y \to Y^*$ defined by
	\begin{equation}\label{T6}
		H(s, x) = \hat Lx +\hat A_t^\vp  x+\hat Cx+t Mx - s\tau_0 j^*v^*, \quad
		s\in [0, 1], \; x\in j^{-1}(\overline {G_1}),
	\end{equation}
	where $t\in (0, t _0]$ is fixed. It can be easily seen that
	$C-s\tau_0 v^*$ is bounded demicontinuous on $\overline {G_1}$ and of type
	$(S_+)$ w.r.t. $D(L)$.

	We now show that the equation $H(s, x) =0$ has no solution on the boundary
	$\partial G_R^1(Y)$. Here, the number $R>0$ is increased, if necessary,  so
	that the ball $B_Y(0, R)$ now also contains all the solutions $x$ of $H(s, x) = 0$.
	To this end, assume the contrary so that there exist $\{t_n\}\subset (0, t_0]$,
	$\{s_n\}\subset [0, 1]$, and $\{x_n\}\subset \partial G_R^1(Y)$ such that
	$t_n\to 0$, $s_n\to s_0$, $x_n\rightharpoonup x_0$ in $Y$,
	$A_{t_n}^\vp x_n\rightharpoonup w^*$ in $X^*,$  $Cx_n\rightharpoonup c^*$ and
	\begin{equation}\label{T7}
		\hat Lx_n +\hat A_{t_n}^\vp  x_n+\hat Cx_n+t_n Mx_n =s_n\tau_0 j^*v^*.
	\end{equation}
	Here, the boundedness of $\{A_{t_n}^\vp x_n\}$ follows as in Step I of
	\cite[Proposition~1]{AK2016}, except that we now use  $A_{t_n}^\vp $ in place of the operators $T_{s_n}$ used in \cite{AK2016}. Since $x_n \rightharpoonup x_0$ in $Y$, we have
	$x_n \rightharpoonup x_0$ in $X$ and $Lx_n \rightharpoonup Lx_0$ in $X^*$.
	Also, since $x_n\in B_Y(0, R)$ and
	$$
	\partial(j^{-1}(G_1)\cap B_Y(0, R)) \subset \partial(j^{-1}(G_1))
	\cup \partial B_Y(0, R) \subset j^{-1}(\partial G_1) \cup \partial B_Y(0, R),
	$$
	we have $x_n\in j^{-1}(\partial G_1) = \partial G_1\cap Y \subset \partial G_1$.
	We now follow the arguments as in \cite[Theorem 2.2]{Adhikari2017} in conjunction with 	Lemma~\ref{L1} to arrive at 
	$$	\langle Lx_0 + w^*+ Cx_0 - s_0\tau_0 v^*, u\rangle = 0	$$
	for all $u\in Y$, where $x_0\in D(A)$ and $w^*\in Ax_0$. Since $Y$ is dense in $X$, we have
	$	Lx_0 + Tx_0+ Cx_0 \ni s_0\tau_0 v^*,$
	which contradicts the hypothesis (H3) because
	$x_0\in D(L)\cap D(T)\cap \partial G_1$.
		
	We shrink $t_0$, if necessary,  so that
	$$	H(s, x) =0,\quad s\in [0, 1], \; x\in \overline{G_R^1(Y)}	$$
	has no solution on the boundary $\partial G_R^1(Y)$ for all $t\in (0, t_0]$
	and all $s\in [0, 1]$.  It now follows from Proposition~\ref{Prop2} that  $H(s, x)$ is an admissible homotopy for
	the $(S_+)$-degree, ${\rm d}_{S_+},$ and therefore
	${\rm d}_{S_+} (H(s, \cdot), G_R^1(Y), 0)$, is well-defined and remains
	constant for all $s\in [0, 1]$. Also, by Proposition~\ref{Prop2},
	the limit  $$ \lim_{t \to0+}{\rm d}_{S_+} (H(1, \cdot),
	G_R^1(Y), 0)
	$$ exists.
	By shrinking $t_0$ further,  if necessary,
	we find that
	$	{\rm d}_{S_+} (H(1, \cdot), G_R^1(Y), 0) = \mbox{a constant}
	\text{ for all } t\in (0, t_0].	$
	Suppose, if possible, that
	$$	{\rm d}_{ S_+} (H(1, \cdot), G_R^1(Y), 0)\ne 0$$ for some $t_1\in (0, t_0]$. Then there exists $x_0\in G_R^1(Y)$ such that
	$$\hat Lx +\hat A_{t_1}^\vp x+\hat Cx+t_1 Mx = \tau_0 j^*v^*.$$ This contradicts the choice of $\tau_0$ as stated in \eqref{T3}. Since
	$${\rm d}_{S_+} (H(0, \cdot), G_R^1(Y), 0)= {\rm d}_{S_+} (H(1, \cdot), G_R^1(Y), 0),$$
	we have
	\begin{equation}\label{D1}
		{\rm d}_{S_+}(\hat L+ \hat A_t^\vp + \hat C + t M, G_R^1(Y), 0)
		={\rm d}_{S_+}(H(0, \cdot), G_R^1(Y), 0) = 0
	\end{equation}
	for all $t\in (0, t_0]$.
		
	Next, we consider the homotopy $\widetilde{H}: [0, 1]\times Y\to Y^*$ defined by
	$$	\widetilde{H}(s, x)= s(\hat Lx +\hat A_t^\vp  x+\hat Cx)
	+t  Mx+ (1-s)\hat Jx, \quad s\in [0, 1], \; x\in j^{-1}(\overline {G_2}).	$$
	As in \cite[Step III, p. 29]{AK2016}, it can be shown that there exists $t_0>0$
	(shrink it to a  smaller number  if necessary) such that all
	the solutions of
	$$
	\widetilde{H} (s, x) = 0, \; t\in (0, t_0], \;s\in [0, 1]
	$$
	are bounded in $Y$. We enlarge the previous number $R>0$, if necessary,
	so that all solutions of $\widetilde{H}(s, x) = 0$ as described above are contained
	in $B_Y(0, R)$ in $Y$.

	Again, by following the arguments similar to that in \cite[Theorem~2.2]{Adhikari2017},    we can show the existence of  $t_1\in (0, t_0]$ such that the equation
	$\widetilde{H}(s, x) = 0$ has no solutions on $\partial G_R^2(Y) $
	for any $t\in (0, t_1]$ and any $s\in [0, 1]$.
	Here, $G_R^2(Y) := j^{-1}(G_2)\cap B_Y(0, R)$. In fact, if we assume the contrary, we can arrive at a situation that contradicts (H4). 
		At this point, we replace the number $t_0$ chosen previously with $t_1$ and
	call it $t_0$ again. Let us fix $t\in (0, t_0]$ and consider the homotopy equation
	\begin{equation}\label{T15}
		\widetilde{H} (s, x) = s(\hat Lx +\hat A_t^\vp  x+\hat Cx)+t Mx
		+ (1-s)\hat Jx =0, \;\; s\in [0, 1], \; x\in \overline {G_R^2(Y)}.
	\end{equation}
	It is already discussed that \eqref{T15} has no solution on $\partial {G_R^2(Y)}$.
	We note that  $\widetilde{H}$ is an affine homotopy of  bounded
	demicontinuous operators of type $(S_+)$ on $\overline {G_R^2(Y)}$;
	namely, $\hat L +\hat A_t^\vp  +\hat C+t M$ and $t M+ \hat J$. We also note here
	that $t M + \hat J$ is strictly monotone. In view of Proposition~\ref{Prop2},  it follows that $\widetilde{H}(s, x)$
	is an admissible homotopy for the $(S_+)$-degree, ${\rm d}_{S_+}$, which satisfies
	\begin{equation}\label{T16}
		{\rm d}_{S_+}(\widetilde{H} (1, \cdot), G_R^2(Y), 0)
		= {\rm d}_{ S_+}(\widetilde{H} (0, \cdot), G_R^2(Y), 0).
	\end{equation}
	This implies
	\begin{equation}\label{D2}
		{\rm d}_{S_+}(\hat L+\hat A_t^\vp +\hat C+t M, G_R^2(Y), 0)
		= {\rm d}_{S_+}(t M+ \hat J, G_R^2(Y), 0)=1
	\end{equation}
	for all $t \in (0, t _0]$.
	The last equality follows from \cite[Theorem 3, (iv)]{BR1983}.
	From \eqref{D1} and \eqref{D2}, we obtain
	$$
	{\rm d}_{S_+}(\hat L+\hat A_t ^\vp +\hat C+t  M, G_R^1(Y), 0)
	\ne {\rm d}_{S_+}(\hat L+\hat A_t^\vp +\hat C+t  M, G_R^2(Y), 0)
	$$
	for all $t \in (0,t_0]$.
	By the excision property of the $(S_+)$-degree, for each $t \in (0, t _0]$,
	there exists a solution $x_t\in G_R^1(Y)\setminus G_R^2(Y)$ of the equation
	$$
	\hat Lx+\hat A_t^\vp x+\hat Cx+t Mx=0.
	$$
	We now pick a sequence $\{t_n\}\subset (0, t_0]$ such that $t_n\to 0$
	and denote the corresponding solution $x_t$ by $x_n$, i.e., 
	$$
	\hat Lx_n+\hat A_{t_n}^\vp x_n +\hat Cx_n+t_n Mx_n=0.
	$$
	Since $Y$ is reflexive, we have $x_n\rightharpoonup x_0\in Y$ by
	passing to a subsequence. This implies $x_n\rightharpoonup x_0$ in $X$ and
	$Lx_n \rightharpoonup Lx_0$ in $X^*$. By the boundedness (therefore strong quasiboundedness) of $A$,
	we may assume, in view of Lemma~\ref{L2}, that $A_{t_n}^\vp x_n\rightharpoonup w^*\in X^*$.
	By a standard argument in conjunction with Lemma~\ref{L1} and the $(S_+)$-property of $C$ w.r.t. $D(L)$, 
	we obtain $x_n\to x_0\in\overline{G_R^1(Y)\setminus G_R^2(Y)}.$ By
	Lemma~\ref{L1} and the demicontinuity of $C$, we have $x_0\in D(A)$, $w^*\in Ax_0$, and $Cx_n \rh Cx_0$ in $X^*.$ Thus,
	$Lx_0 +Ax_0+Cx_0 \ni 0$. 

Finally, in order to show $x_0\in G_1\setminus G_2$, we note that
	\[
	G_R^1(Y)\setminus G_R^2(Y) =(G_1\setminus G_2) \cap Y\cap B_Y(0, R)
	\subset G_1\setminus G_2.
	\]
Consequently, 
	$x_n\in G_1\setminus G_2$ for all $n$, and therefore
	$$
	x_0\in\overline{G_1\setminus G_2} \subset (G_1\setminus G_2)
	\cup \partial (G_1\setminus G_2)\subset (G_1\setminus G_2)
	\cup \partial G_1 \cup \partial G_2.
	$$
	By   (H3) and (H4), $x_0\not\in \partial G_1\cup\partial G_2$  and hence $x_0\in D(L)\cap D(T)\cap (G_1\setminus G_2)$.
\end{proof}

\subsection{Open Problem} Does Theorem~\ref{Th2} hold true if the boundedness of  $A$ is dropped? 
Since a positively homogeneous operator that is  strongly quasibounded  is necessarily bounded,  it is desirable to determine whether Theorem~\ref{Th2} holds  if  $A$ is assumed to be ``quasibounded".  	An operator $A:X \supset D(A) \to 2^{X^*}$ is said to be \textit{quasibounded} if for every $S>0$ there exists $K(S)>0$ such that
$\|x\| \le S$ and $\langle x^*, x \rangle \le S\|x\|$ for some $x^*\in Ax$
imply $\|x^*\| \le K(S)$. The notions of quasibounded and strongly quasibounded operators  were introduced in Hess~\cite{Hess}. 

\section{Applications}\label{S4}
In this section, we apply Theorem~\ref{Th1} and Theorem~\ref{Th2}  to elliptic and parabolic boundary value problems in general divergence form which are obtained  by modifying relevant examples  from Berkovits and Mustonen ~\cite{BM},  Kittil\"a ~\cite{Kittila},    and Adhikari~\cite{Adhikari2017}.

\begin{application}{\rm 
We consider the space $X= W_0^{m,p}(\Omega)$ with the
integer $m\ge 1$, the number $p\in(1,\infty)$, and the
domain $\Omega \subset \mathbf{R}^N$  with smooth boundary. We let $N_0$ denote
the number of all multi-indices
$\alpha=(\alpha_1,\dots,\alpha_N)$ such that
$|\alpha | = \alpha_1 +\cdots +\alpha_N\le m$.
For $\xi = (\xi_\alpha)_{|\alpha|\le m}\in\mathbf{R}^{N_0}$,
we have a representation $\xi=(\eta,\zeta)$, where
$\eta=(\eta_\alpha)_{|\alpha|\le m-1}\in\mathbf{R}^{N_1}$,
$\zeta =(\zeta_\alpha)_{|\alpha|=m}\in\mathbf{R}^{N_2}$
and $N_0=N_1+N_2$. We let
$$
\xi(u)= (D^\alpha u)_{|\alpha|\le m} ,\quad
\eta(u)= (D^\alpha u)_{|\alpha|\le m-1}, \quad \mbox{and}\quad 
\zeta(u)= (D^\alpha u)_{|\alpha|= m},
$$
where
$
D^\alpha  =\prod_{i=1}^N\Big(\frac{\partial}{\partial  x_i}\Big)^{\alpha_i}.
$
We write $\nabla u := (D^\alpha u)_{|\alpha|= 1}$, and when   $|\alpha| = k\in\{1, 2, \dots, m\}$, we simply  write $D^k u:=	(D^\alpha u)_{|\alpha|=k}$.	Also, define $q:= p/(p-1)$.

We now consider the partial differential expression in divergence form
$$
\sum_{|\alpha|\le m}(-1)^{|\alpha|}
D^\alpha A_\alpha(x, \xi(u)),\quad x\in\Omega.
$$
The functions $A_\alpha :\Omega\times\mathbf{R}^{N_0}\to \mathbf{R}$
are assumed to be Carath\'eodory, i.e.,
each $A_\alpha(x, \xi)$ is measurable in $x$ for fixed
$ \xi\in\mathbf{R}^{N_0}$ and continuous in $\xi$ for
almost all $x\in\Omega$.
We assume the following conditions on $A_\alpha$:
\begin{itemize}
	\item[(H5)] There exist $p\in(1,\infty),$ $c_1 >0,$
	and $\kappa_1\in L^q(\Omega)$ such that
	$$
	|A_\alpha(x, \xi)|\le c_1|\xi|^{p-1}+
	\kappa_1(x),\quad x\in\Omega,\;\;\xi\in\mathbf{R}^{N_0},\;\;|\alpha| \le m.
	$$
	\item[(H6)] The Leray-Lions condition
	$$
	\sum_{|\alpha|=m} [A_\alpha(x, \eta, \zeta_1)-
	A_\alpha(x, \eta, \zeta_2)](\zeta_{1_\alpha}-\zeta_ {2_\alpha})>0
	$$
	is satisfied for every $x\in \Omega$, $\eta\in\mathbf{R}^{N_1}$ and 
	$\zeta_1, \zeta_2\in\mathbf{R}^{N_2}$ with
	$\zeta_1\ne \zeta_2$.
	\item[(H7)]
	$$
	\sum_{|\alpha|\le m} [A_\alpha(x, \xi_1)-
	A_\alpha(x, \xi_2)](\xi_{1_\alpha}-\xi_ {2_\alpha})\ge0
	$$
	is satisfied for every $x\in \Omega$ and
	$\xi_1, \xi_2\in\mathbf{R}^{N_0}$.
	\item[(H8)] There exist $c_2>0$, $\kappa_2\in
	L^1(\Omega)$ such that
	$$
	\sum_{|\alpha|\le m}A_\alpha(x,\xi)\xi_\alpha \ge
	c_2|\xi|^p-\kappa_2(x),\quad x\in\Omega,\;
	\xi\in\mathbf{R}^{N_0}.
	$$
	\item[(H9)]  Each $A_\alpha(x,\xi)$ is homogeneous of degree $\gamma>0$ w.r.t. $\xi$.  
\end{itemize}
If an operator $A: W_0^{m,p}(\Omega)\to W^{-m, q}(\Omega)$ is
given by
\begin{equation} \label{176}
	\langle Au, v\rangle = \int_\Omega\sum_{|\alpha|\le m}
	A_\alpha(x, \xi(u))D^\alpha v,\quad u,\,v\in
	W_0^{m,p}(\Omega),
\end{equation}
then the conditions (H5), (H7) imply that $A$ is bounded,
continuous, and monotone  as discussed in Kittil\"a  
\cite[pp. 25-26]{Kittila} and Pascali and Sburlan
\cite[pp. 274-275]{PS}. Since $A$ is continuous,
it is maximal monotone. 
Moreover,  the condition (H9) implies that $A$ is  positively homogeneous of degree $\gamma>0$.
For example,  for  $m=1$, we have $|\alpha| \le 1$, and when  
$$A_\alpha(x,\eta,\zeta) =\begin{cases}
	|\zeta|^{p -2}\zeta_\alpha&\mbox{ for } |\alpha| = 1\\
	0 &\mbox{ for } |\alpha| = 0,\\
\end{cases}
$$
the operator $A$ in (\ref{176}) is given by $A := -\Delta_p$, where  $\Delta_p $ is the $p-$Laplacian  from $W_{0}^{1, p}(\Omega)$ to $W^{-1, q}(\Omega)$
defined as
$$\Delta_p u:=  {\rm div}\left(|\nabla u|^{p-2} \nabla u\right),\quad u\in W_{0}^{1, p}(\Omega).$$
It is clear that $ \Delta_p$ is positively homogeneous  of degree $p-1$ .

Similarly, the condition (H5), with $A_\alpha$ replaced by $C_\alpha$,
implies that the operator
\begin{equation} \label{177}
	\langle Cu, v\rangle
	= \int_\Omega\sum_{|\alpha|\le m} C_\alpha(x, \xi(u))D^\alpha v,\quad\quad u,\,v\in
	W_0^{m,p}(\Omega),
\end{equation}
is a bounded continuous mapping. We also know that
conditions (H5), (H6), and (H8), with $C_\alpha$ in place
of $A_\alpha$ everywhere, imply that the operator $C$ is of
type $(S_+)$ (see Kittil\"a \cite[ p. 27]{Kittila}).

We also consider a multifunction $H:\Omega\times \mathbf{R}^{N_1}\to 2^{\mathbf{R}}$
such that
\begin{itemize}
	\item [(H10)] $H(x, r) = [\varphi(x, r), \psi(x, r)]$ is measurable in $x$
	and upper semicontinuous in $r$, where $\varphi, \psi :\Omega\times \mathbf{R}^{N_1}\to \mathbf R$
	are measurable functions; and
	\item[(H11)] $|H(x, r)| = \max[|\varphi(x, r)|, |\psi(x, r)|] \le a(x) + c_2|r| $
	a.e. on $\Omega\times\mathbf{R}^{N_1},$  where $a(\cdot) \in L^q(\Omega)$, $c_2>0$.
\end{itemize}
Define $T:W_0^{m,p}\to 2^{W^{-m, q}(\Omega)}$
by
\begin{align*}
	Tu= \Big\{& h\in W^{-m, q}(\Omega) : \exists w\in L^q(\Omega) \text{ such that }
	w(x)\in H(x, u(x)) \\
	&\text{ and }  \langle h, v\rangle = \int_{\Omega} w(x) v(x)
	\text{ for all } v\in W_0^{m, p}(\Omega)\Big\}.
\end{align*}
It is well-known that $T$ is upper-semicontinuous and compact with closed and convex
values  (see \cite[p. 254]{HP}), and therefore $T$ is of class $(P)$.

We now state the following theorem as an application of Theorem~\ref{Th1}.
\begin{theorem}\label{Th3}
	Assume that the operators $A$, $C,$ and $T$ are defined as above. Assume, further, that the rest of the conditions
	of Theorem~\ref{Th1} are satisfied for two
	balls $G_1 = B_{\delta_1}(0)$ and $G_2= B_{\delta_2}(0)$, where $0<\delta_2<\delta_1$. Then
	the Dirichlet boundary value problem
\[\begin{cases}\displaystyle 
		\sum_{|\alpha|\le m}(-1)^{|\alpha|}
D^\alpha \Big(A_\alpha(x, \xi(u))+
C_\alpha(x, \xi(u))\Big)+ H(x, u) \ni 0,\; x\in\Omega,\\
	D^\alpha u(x) = 0,\quad
x\in\partial\Omega,\quad |\alpha| \le m-1,
\end{cases}
 \]
	has a ``weak" nonzero solution
	$u\in B_{\delta_1}(0) \setminus B_{\delta_2}(0)\subset W_0^{m,p}(\Omega)$, which satisfies the
	inclusion $Au + Cu +Tu\ni 0$.
\end{theorem}}

\end{application}

\begin{application}
	{\rm
Let $\Omega$ be a bounded open set in $\mathbf{R}^N$ with smooth boundary,
$m\ge 1$ an integer, and $a>0$. Set $Q= \Omega\times [0, a]$.
Consider differential operators of the form
\begin{equation}\label{IV1}
			\frac{\partial u}{\partial t}(x,t) 	+\sum_{|\alpha|\le m}(-1)^{|\alpha|}D^\alpha \Big(A_\alpha(x,t,\xi(u(x,t))\\
			+C_\alpha(x,t,\xi(u(x,t))\Big)\\
\end{equation}
in $Q$. The functions $A_\alpha=A_\alpha
(x,t,\xi)$ and $C_\alpha=C_\alpha
(x,t,\xi)$ are defined for $(x,t)\in Q$,
$\xi=(\xi_\alpha)_{ |\alpha|\le m}=(\eta,\zeta)\in\mathbf{R}^{N_0}$ with
$\eta=(\eta_\gamma)_{ |\alpha|\le m-1}\in\mathbf{R}^{N_1}$,
$\zeta=(\zeta_\alpha)_{ |\alpha|=m}\in\mathbf{R}^{N_2}$, and $N_1+N_2 = N_0$.
We assume that each  $A_\alpha(x,t,\xi)$ satisfies the usual
Carath\'eodory condition. We consider the following conditions.
\begin{itemize}
	\item[(H12)] (Continuity) For some $p > 1$, $c_1>0$, $g\in L^q(Q)$ with
	$q=p/(p-1)$, we have
	\[
	|A_\alpha(x,t,\eta,\zeta)| \le c_1(|\zeta|^{p-1}+|\eta|^{p-1}+g(x,t)),
	\]
	for $(x,t)\in Q$, $\xi=(\eta,\zeta)\in\mathbf{R}^{N_0}$  and $|\alpha |\le m$.
	
	\item[(H13)] (Monotonicity)
	$$
	\sum_{|\alpha| \le m}(A_\alpha(x,t,\xi_1)-
	A_\alpha(x,t,\xi_2))(\xi_{1_\alpha}-\xi_{2_\alpha}) \ge 0 \mbox{ for } 
	(x,t)\in Q \mbox{ and } \xi_1,\xi_2\in\mathbf{R}^{N_0}.
	$$
	
	\item[(H14)] (Leray-Lions)
	\[
	\sum_{|\alpha| = m}(A_\alpha(x,t,\eta,\zeta)-
	A_\alpha(x,t,\eta,\zeta^*))(\zeta_\alpha-\zeta^*_\alpha) > 0,
	\]
	for $(x,t)\in Q$, $\eta\in\mathbf{R}^{N_1}$ and $\zeta,\zeta^*\in\mathbf{R}^{N_2}$.
	
	\item[(H15)] (Coercivity) There exist $c_0>0$ and $h\in L^1(Q)$ such that
	$$
	\sum_{|\alpha|\le m}A_\alpha(x,t,\xi) \ge c_0|\xi|^p-h(x,t),\quad (x,t)\in Q \mbox{ and }
	\xi\in\mathbf{R}^{N_0}.
	$$
	\item[(H16)] 	 Each $A_\alpha(x, t, \xi)$ is homogeneous of degree $\gamma>0$ w.r.t. $\xi$.  
\end{itemize}

Under the condition (H12), the second term of \eqref{IV1} with $C_\alpha =0$ generates a
continuous bounded operator
$ A:X\to X^*$ 
defined by
$$
\langle Au,v\rangle=\sum_{|\alpha|\le m}\int_QA_\alpha(x,t,\xi(u(x, t)))D^\alpha v,
\quad u,v\in X,
$$ where $X=L^p(0,a;V), X^*=L^q(0,a;V^*)$,
and $V=W_0^{m,p}(\Omega).$ 
With the additional conditions (H13) and (H16), the operator  $A$ is  maximal monotone  and positively homogeneous of degree $\gamma.$  
Under (H12), (H14), and (H15) with $A_\alpha$ replaced by $C_\alpha$ and  other obvious
changes,  the second term in \eqref{IV1} with $A_\alpha = 0$ generates a continuous, bounded
operator $C$  defined as 
$$
\langle Cu,v\rangle=\sum_{|\alpha| \le m}\int_QC_\alpha(x,t,\xi(u(x, t)))D^\alpha v,
\quad u,v\in X,
$$
which satisfies the condition $(S_+)$ w.r.t. $D(L)$, where the
operator $L$ is defined as follows.  The operator $\partial/\partial t$ generates an operator $L:X\supset D(L)\to X^*$,
where
$$
D(L) = \{v\in X: v'\in X^*,\; v(0)=0\},
$$
via the relation
$$
\langle Lu,v\rangle = \int_0^a\lan u'(t),v(t)\ran_V \,\text{d}t,\quad u\in D(L),\; v\in X,
$$
where   $\lan \cdot, \cdot \ran_V$ is the duality pairing in $V^*\times V.$
The symbol $u'(t)$  is the generalized derivative of $u(t)$, i.e., 
$$
\int_0^a  u'(t) \varphi(t)\text{d}t
=-\int_0^a  \varphi'(t)  u(t) \,\text{d}t,\quad \varphi\in C_0^\infty(0,a).
$$
We can verify, as in Zeidler \cite{Zeidler1}, that $L$ is   densely
defined, linear and  maximal monotone.

 Given $h\in L^q(Q),$ define  $h^*\in X^*$  by
 \[\lan h^*, v\ran =\int_Q h v, \quad v\in X.\]
 }
\end{application}
As an application of Theorem~\ref{Th2}, we obtain the following theorem.

\begin{theorem} \label{thm4}
	Assume that the operators $L, A,$ and $C$ are as above, with
	$A_\alpha$ satisfying {\rm (H12), (H13),} and {\rm (H16)},  and
	$C_\alpha$ in place of $A_\alpha$ satisfying {\rm (H12), (H14),}
	and {\rm (H15)}.  Assume, for a given $h \in L^q(Q), $   that the
	rest of the conditions of Theorem~\ref{Th2} are satisfied  when $C$ is replaced with $C-h^*$  for two
	balls $G_1 = B_{\delta_1}(0)$ and $G_2= B_{\delta_2}(0)$ in $X = L^p(0, a; V)$, where
	$0<\delta_2<\delta_1$ and $V= W_0^m(\Omega)$. 
	 Then	the initial-boundary value problem
	\[\begin{cases}\displaystyle
					\frac{\partial u}{\partial t}
		+\sum_{|\alpha|\le m}(-1)^{|\alpha|}D^\alpha \Big(A_\alpha(x,t,\xi(u))+C_\alpha(x,t,\xi(u))\Big)=h(x, t),\\
				D^\alpha u(x, t) = 0,\quad
		(x, t)\in\partial\Omega\times [0, a],\quad |\alpha| \le m-1,\\
		u(x,0) = 0, \quad x\in \Omega,
	\end{cases}\]
has a ``weak" nonzero solution $u\in B_{\delta_1}(0) \setminus B_{\delta_2}(0) \subset L^p(0, a;V)$ satisfying
	$$
	Lu + Au + Cu = h^*.
	$$
	
\end{theorem}

\subsection*{Acknowledgments}
This research is carried out by  members of  the Analysis Group of the Association of Nepalese Mathematicians in America (ANMA) within  the \textit{Collaborative Research in Mathematical Sciences} program. Ghanshyam Bhatt acknowledges Tennessee State University for supporting his engagement  in this research with a Non-Instructional Assignment (NIA) grant for the years 2021-2022. The authors express their gratitude to  the anonymous referees for providing  feedback which helped to improve the paper.

\end{document}